\documentclass[11pt,a4paper,reqno]{amsart}

\usepackage[margin=30mm,headheight=14pt,headsep=16pt]{geometry}
\usepackage[T1]{fontenc}
\usepackage{lmodern}
\usepackage{microtype}

\allowdisplaybreaks[1]
\usepackage{amsmath,amsthm,amssymb}

\usepackage{enumitem}
\setlist[enumerate]{itemsep=3pt,topsep=5pt,label=\emph{(\alph*)}}
\setlist[enumerate,2]{itemsep=2pt,label=\emph{(\roman*)}}

\usepackage{float}
\usepackage[section]{placeins}
\usepackage{tikz}

\usepackage{xcolor}
\definecolor{mydarkblue}{RGB}{35,64,90}

\usepackage[hypertexnames=false]{hyperref}
\usepackage{bookmark}
\hypersetup{
  pdftitle={Graded Representation Theory of Equivariant Neural Networks},
  pdfauthor={Mani Shayestehfar},
  colorlinks=true,
  linkcolor=mydarkblue,
  citecolor=mydarkblue,
  urlcolor=mydarkblue,
  linktocpage=true
}

\usepackage[nameinlink,noabbrev]{cleveref}
\crefname{section}{Section}{Sections}
\crefname{subsection}{Section}{Sections}
\crefname{appendix}{Appendix}{Appendices}

\usepackage{aliascnt}

\newcommand{\sharedtheorem}[3]{%
  \newaliascnt{#1}{theorem}%
  \newtheorem{#1}[#1]{#2}%
  \aliascntresetthe{#1}%
  \crefname{#1}{#2}{#3}%
}

\theoremstyle{plain}
\newtheorem{theorem}{Theorem}[section]
\crefname{theorem}{Theorem}{Theorems}
\sharedtheorem{proposition}{Proposition}{Propositions}
\sharedtheorem{lemma}{Lemma}{Lemmas}
\sharedtheorem{corollary}{Corollary}{Corollaries}

\theoremstyle{definition}
\sharedtheorem{definition}{Definition}{Definitions}
\sharedtheorem{example}{Example}{Examples}

\theoremstyle{remark}
\sharedtheorem{remark}{Remark}{Remarks}
\sharedtheorem{question}{Question}{Questions}

\DeclareMathOperator{\Sym}{Sym}
\DeclareMathOperator{\Hom}{Hom}
\DeclareMathOperator{\End}{End}
\DeclareMathOperator{\supp}{supp}
\DeclareMathOperator{\tr}{tr}
\DeclareMathOperator{\Ind}{Ind}
\DeclareMathOperator{\Aut}{Aut}
\DeclareMathOperator{\Cyc}{Cyc}
\DeclareMathOperator{\MSet}{MSet}
\DeclareMathOperator{\Poly}{Poly}
\DeclareMathOperator{\Span}{span}

\newcommand{\E}{\mathbb E}
\newcommand{\R}{\mathbb R}
\newcommand{\C}{\mathbb C}
\newcommand{\N}{\mathbb N}
\newcommand{\Z}{\mathbb Z}

\newcommand{\cH}{\mathcal H}
\newcommand{\cB}{\mathcal B}
\newcommand{\cO}{\mathcal O}
\newcommand{\one}{\mathbf 1}

\newcommand{\PL}{\mathrm{PL}}
\newcommand{\He}{\mathrm{He}}
\newcommand{\whbigoplus}{\widehat{\bigoplus}}
\newcommand{\ip}[2]{\left\langle #1,#2\right\rangle}
\newcommand{\norm}[1]{\left\lVert #1\right\rVert}
\newcommand{\abs}[1]{\left\lvert #1\right\rvert}

\newcommand{\Vnr}{V_{n,r}}
\newcommand{\Vns}{V_{n,s}}
\newcommand{\Xnr}{X_{n,r}}
\newcommand{\Xns}{X_{n,s}}

\title[Graded Representation Theory of Equivariant Neural Networks]{Graded Representation Theory of Equivariant Neural Networks}

\author{Mani Shayestehfar}
\address{School of Mathematics and Statistics, University of Sydney, NSW 2006, Australia}
\email{mani.shayestehfar@sydney.edu.au}

\urladdr{https://maani.info}

\date{}

\begin{document}

\begin{abstract}
Nonlinear activations can create equivariant interactions between
irreducible representations that linear maps cannot. We use the
Gaussian degree decomposition to extend ordinary polynomial degree to
such nonlinear maps, and prove that for a fixed coordinatewise
equivariant layer each degree factors into a polynomial determined by
the linear maps and a scalar determined by the activation. This
separates three distinct obstructions, coming from symmetry,
coordinates, and activation.
\end{abstract}

\maketitle
\tableofcontents

\section*{Introduction}\label{sec:intro}

Neural networks alternate linear maps with nonlinear activation functions. When the data carry symmetries, one often asks the network to respect them: rotating, translating or relabelling an input should have a prescribed effect on the output. Mathematically, the feature spaces are then representations of a group $G$, and the linear parts of the network are $G$-equivariant maps. This is precisely the viewpoint promoted by geometric deep learning \cite{gdl}. 

Representation theory is very effective for understanding the linear part. A finite-dimensional representation can be decomposed into irreducible pieces, and an equivariant linear map cannot mix two non-isomorphic irreducibles. In a suitable basis, this makes the linear maps much simpler.

The activation function changes the picture. It is nonlinear and is usually applied coordinatewise, so it can create interactions between irreducible pieces which are completely invisible to linear representation theory. The simplest example already occurs for the symmetric group $S_3$.

Let $S_3$ permute the coordinates of
$$
\R^3=L_0\oplus L_1,
\qquad
L_0=\Span\{(1,1,1)\},
\qquad
L_1=\{x:x_1+x_2+x_3=0\}.
$$
The first summand is the trivial representation and the second is the two-dimensional standard representation. There is no nonzero equivariant linear map from $L_1$ to $L_0$. Coordinatewise squaring, however, gives
\begin{equation}\label{eq:intro-square}
p_0(x_1^2,x_2^2,x_3^2)
=
\frac{\norm{x}^2}{3}(1,1,1),
\qquad x\in L_1,
\end{equation}
where $p_0$ is orthogonal projection onto $L_0$. Thus a coordinatewise nonlinearity creates a quadratic interaction between two representation types which no linear equivariant map can connect.

This is the starting point of the paper. We ask not merely whether a nonlinear equivariant interaction exists, but at which degrees it occurs and whether a given neural network layer actually produces it.

\subsection*{From existence to degree}

Gibson, Tubbenhauer and Williamson \cite{GTW2025} (GTW) study the first, degree-free version of this question. For irreducible real representations $L$ and $K$ of a finite group, they show that a nonzero equivariant piecewise-linear map $L\to K$ exists precisely when
\begin{equation}\label{eq:gtw-intro}
\ker L\subseteq\ker K.
\end{equation}
Here piecewise linear means continuous and affine on each of finitely many polyhedral pieces. This is already very different from the linear theory: nonlinear equivariant maps can connect many different irreducible representations.

Condition \eqref{eq:gtw-intro}, however, does not say how the interaction occurs. For example, the interaction in \eqref{eq:intro-square} first appears in degree two. More generally, one would like to distinguish such interactions by the degree in which they occur. One would also like to distinguish two further questions: symmetry may permit a degree without the chosen coordinate system producing it, and the coordinate system may produce it while the chosen activation suppresses it.

For polynomial maps there is an obvious notion of degree: the degree-$k$
equivariant maps $L\to K$ are
$$
\Hom_G(\Sym^k L,K).
$$
Activations such as ReLU are not polynomial, so we need a way to extend
this notion of degree. We put the standard Gaussian measure on $L$.
Just as monomials organise ordinary polynomials by degree, Hermite
polynomials form an orthogonal basis for functions with respect to the
Gaussian measure. Grouping the multivariate Hermite polynomials by their
total degree gives an orthogonal decomposition, called the \emph{polynomial chaos}
decomposition.

Thus polynomial chaos extends the usual polynomial grading to non-polynomial activations such as ReLU. We write $\Pi_k$ for the orthogonal projection onto the
degree-$k$ part.

This grading is classical. Its usefulness here is that it gives a common language for polynomial and non-polynomial activations. A ReLU layer, for example, has infinitely many Hermite components, and we can ask exactly which representation-theoretic interactions occur in each one.

There is already a useful answer before considering a particular architecture. If $L\neq0$, then an allowed degree can never occur in isolation:
$$
\Hom_G(\Sym^kL,K)\neq0
\quad\Longrightarrow\quad
\Hom_G(\Sym^{k+2m}L,K)\neq0
\qquad(m\geq0).
$$
Thus the possible degrees consist of at most two tails, one even and one odd. For finite groups we also refine \eqref{eq:gtw-intro}: writing $N=\ker L$, some polynomial degree occurs if and only if $K^N\neq0$ (where $K^N$ denotes the $N$-fixed subspace of $K$), and the least such degree is at most $|G/N|-1$. For irreducible $K$, this recovers the GTW kernel condition and additionally locates the interaction in degree.

\subsection*{What does one layer produce?}

Knowing that a degree-$k$ equivariant map $L\to K$ exists does not mean
that a given neural network layer produces one. We now make this
distinction precise.

Let $H=\R^{\cB}$ be a representation with a basis $\cB$ which is permuted by $G$.
A typical equivariant neural network layer has the form
$$
L \xrightarrow{A} H
  \xrightarrow{\sigma_{\cB}} H
  \xrightarrow{R} K.
$$
The first map $A$ expresses the input in the hidden coordinates $\cB$,
and the last map $R$ sends the hidden features to the output space $K$.
Between them we apply the same scalar activation
$\sigma:\R\to\R$ separately to each hidden coordinate: if
$$
h=\sum_{b\in\cB} h_b\,b, \qquad \text{then} \qquad \sigma_{\cB}(h)=\sum_{b\in\cB}\sigma(h_b)\,b.
$$
Since $G$ only permutes the basis vectors in $\cB$, applying the same
function to every coordinate commutes with the $G$-action. Hence
$\sigma_{\cB}$ is equivariant, and so is the whole layer whenever
$A$ and $R$ are equivariant linear maps.

There are two separate reasons why a degree permitted by representation
theory may fail to appear in this layer:
\begin{enumerate}
    \item the linear maps $A$ and
$R$ may not produce the relevant polynomial interaction, or
\item the
activation $\sigma$ may have no component in that degree.
\end{enumerate}
The main calculation separates these two effects. Let $s^2$ be
the common variance of the coordinates of $AX$ when $X$ is a standard
Gaussian vector in $L$. Let $a_k(\sigma;s)$ be the coefficient of the
$k$th Hermite polynomial in the expansion of the one-variable function
$z\mapsto \sigma(sz)$. Finally, for an ordinary input $x\in L$, set
$$
Q_k(x)=R\bigl((Ax)^{\odot k}\bigr),
$$
where $(Ax)^{\odot k}$ is the coordinatewise $k$th power. Then
\begin{equation}\label{eq:intro-main}
\Pi_k\bigl[R\sigma_{\cB}(Ax)\bigr]
=
\frac{a_k(\sigma;s)}{s^k}\,\Pi_k Q_k.
\end{equation}

This gives three separate questions:
\begin{enumerate}
    \item \label{question:symmetry}
    Does symmetry allow a degree-$k$ equivariant map $L\to K$?
    \item \label{question:linear-maps}
    Do the chosen linear maps $A$ and $R$ actually produce one?
    \item \label{question:activation}
    Does the activation $\sigma$ retain degree $k$?
\end{enumerate}
The first asks whether $\Hom_G(\Sym^k L,K)\neq 0$.
The second asks whether $Q_k$ is nonzero, and the third whether $a_k(\sigma;s)$ is nonzero.

The second question also has an explicit answer. Put $C=AA^*$, written
in the coordinate basis of $\R^{\cB}$ (with  $G$ acting transitively). Then
\begin{equation}\label{eq:intro-norm}
\norm{\Pi_kQ_k}_{L^2}^2
=
k!\,\tr\bigl(RC^{\circ k}R^*\bigr),
\end{equation}
where $C^{\circ k}$ is obtained by taking the $k$th power of each entry
of $C$. Hence
$$
Q_k\neq0
\quad\Longleftrightarrow\quad
\tr\bigl(RC^{\circ k}R^*\bigr)>0.
$$

For $S_4$ we give an
example where a quadratic equivariant map exists, but the chosen
coordinatewise layer produces no quadratic term at all, and changing the
activation cannot fix this. For $C_8$ the opposite phenomenon occurs:
the layer does produce the first permitted polynomial interaction, but
a positively homogeneous activation such as ReLU removes it. Adding a
bias restores it.

Thus a nonlinear interaction can fail for three different reasons: it
may be forbidden by symmetry, missed by the chosen linear maps, or
removed by the activation. Distinguishing these three mechanisms is
the main point of the paper.

\subsection*{Consequences for architectures}

The same calculation answers two practical questions. For any fixed
collection of layers, an explicit Gram matrix determines exactly which
degree-$k$ equivariant maps they can produce. Moreover, shifted ReLUs
can realise arbitrary Hermite coefficients up to any prescribed finite
degree. Using the regular representation then gives all equivariant
polynomial directions, although not efficiently.

For ordinary ReLU there is a simple restriction: all odd Hermite degrees
above one vanish. A shared bias removes this restriction generically.
These degree-resolved interactions can also be used to design
representation-aware architectures.
In a companion paper by the author, the resulting source-to-target
decomposition is used to build a higher-order graph network on unordered
triples, and experiments show an improvement over both a standard GNN
baseline and a parameter-matched control with the full linear
$S_n$-equivariant mixing space \cite{CompanionPaper}. This addresses a question
left open in \cite{GTW2025}: whether organising nonlinear equivariant
networks by irreducible representation theory is useful in practice.

\subsection*{Subset features}

The final part of the paper treats the action of $S_n$ on $r$-element
subsets. A degree-$d$ monomial in these variables corresponds to an $r$-uniform
multihypergraph with $d$ edges. If the output is indexed by
$s$-subsets, one simply marks an $s$-subset of vertices. Equivariant
polynomial maps are therefore indexed by these marked multihypergraphs,
up to relabelling.

This gives an explicit basis whose types stabilise at the sharp bound
$$
n=rd+s.
$$
For quadratic edge-to-vertex maps there are seven stable types, and all
seven can be evaluated in $O(n^2)$ operations, or $O(n+m)$ for $m$
nonzero edges.

\subsection*{Relation to existing work and scope}
\label{sec:relation}

Our starting point is the piecewise-linear representation theory of \cite[Theorem~2H.3]{GTW2025}. For irreducible real representations of a finite group, they characterise the existence of a nonzero equivariant piecewise-linear map by inclusion of the representation kernels. We refine this existence question by recording the polynomial degrees in which an interaction can occur. We then ask
which of these interactions are produced by a specified
coordinatewise layer and activation.

The compact-group structure theorem of
\cite[Theorem~2.5 and Remark~2.6]{Aiello2026}
gives a broader context for this finite-group analysis.
Let $\Gamma$ be a compact group with identity component
$\Gamma^\circ$, and equip its finite-dimensional real
representations $L$ and $K$ with invariant inner products.
Every equivariant piecewise-linear map $F:L\to K$ has the form
$$
F(v+w)=\psi(v)+Tw,
\qquad
v\in L^{\Gamma^\circ},
\quad
w\in (L^{\Gamma^\circ})^\perp,
$$
where $\psi:L^{\Gamma^\circ}\to K^{\Gamma^\circ}$ is
equivariant and piecewise linear, and $T:(L^{\Gamma^\circ})^\perp \to (K^{\Gamma^\circ})^\perp$ is an equivariant linear map.
The joint action on the two fixed subspaces has finite image. Consequently, our finite-group degree criteria apply to this remaining nonlinear problem, and our layer criteria can test chosen coordinatewise constructions for $\psi$.
For connected groups, nonlinear dependence is confined to the trivial subrepresentations. This restriction also explains why polynomial equivariants alone cannot characterise piecewise-linear existence for general compact groups.

At the level of layer construction, \cite{Pacini2024} classify the combinations of representations, coordinates and pointwise activations that give an equivariant activation map. We work with hidden permutation representations, where shared coordinatewise activations are equivariant, and analyse the interactions selected by the linear lift, activation and readout.

The analytic tools are classical. Gaussian chaos, Wick ordering and their interpretation through symmetric tensors are developed in \cite[Chapters~2--4]{Janson1997}. Hermite expansions and the associated covariance identities also underlie the dual-activation framework of \cite{Daniely2016}. We use these identities to separate the polynomial selected by an equivariant layer from the scalar coefficient supplied by its activation. This gives explicit Gram criteria for the vanishing of a degree component and for the degree subspace spanned by a specified family of layers.

Our spanning construction also connects with equivariant
approximation theory. \cite{Ravanbakhsh2020}
establishes universality using regular hidden actions, while
\cite{Yarotsky2022} obtains invariant and equivariant approximation through group averaging. Here regular
representations provide an exact spanning construction in each polynomial degree. Matching finitely many Hermite components does not, by itself, control the remaining approximation error.

For subset representations, the marked-hypergraph basis follows the orbit-sum principle used by \cite{Puny2023} to construct multigraph-indexed bases of equivariant graph polynomials. We formulate this construction for $r$-subset inputs and $s$-subset outputs, and determine its sharp stable range. The distinction between parameter sharing across dimensions and compatibility with embeddings already appears in \cite{LevinDiaz2024}; our marked-orbit criterion
makes that compatibility explicit for these polynomial maps.

The Gaussian measure here is only a reference
measure for the grading, not an assumption about trained features, and the resulting one-layer obstructions need not persist under arbitrary composition.

\subsection*{Organisation}\label{sec:organisation}

\Cref{sec:prelim,sec:gaussian} introduce polynomial and Gaussian degree and prove the finite-group degree-detection results. \Cref{sec:coordinatewise} studies coordinatewise layers and proves \eqref{eq:intro-main} and \eqref{eq:intro-norm}. \Cref{sec:activation-spectra} determines the relevant activation spectra, and \cref{sec:architecture} gives the resulting architecture-level spanning statements. \Cref{sec:subsets} treats polynomial layers on subset features. Supplementary proofs and representation formulae are collected in the appendices.

\subsection*{Acknowledgement}
This research was completed during a master's project supervised by Dani Tubbenhauer, whom I thank for many discussions and for guidance throughout.

\subsection*{Use of AI}
The initial backbone of the work was developed independently. Discussions with Claude Fable pointed towards sources such as \cite{Janson1997}, where the connection between $L^2(\gamma_L)$ and $\Sym^kL$ was found. ChatGPT Sol assisted with polishing arguments and proofs, and with the literature review. After using each tool, the author reviewed and edited the content as needed, and takes full responsibility for the content of the work.

\vspace{0.7cm}
\section{Preliminaries}\label{sec:prelim}

All representations are finite-dimensional and real unless otherwise stated, and are equipped with fixed invariant inner products. The acting group $G$ is finite unless compact groups are explicitly allowed. For a representation $L$, we write $\ker L$ for the kernel of the action; it is a normal subgroup of $G$.

\subsection{Polynomial maps and symmetric powers}
Polynomial degree is the basic grading used throughout the paper. We begin by recalling the standard identification between homogeneous polynomial maps and maps out of symmetric powers.

For a real representation $L$ and $k\ge1$, define
$$
\Sym^k L
=
L^{\otimes k}\Big/
\Span\{t-\pi t:\ t\in L^{\otimes k},\ \pi\in S_k\},
\qquad
\Sym^0L=\R,
$$
where $S_k$ acts by permuting tensor factors. The diagonal action of $G$ on $L^{\otimes k}$ commutes with this action, and therefore descends to $\Sym^kL$.

For $x\in L$, we write $x^k$ for the class of $x^{\otimes k}$ in $\Sym^kL$. Pure powers span $\Sym^kL$, and the induced inner product is characterised by
$$
\ip{x^k}{y^k}=\ip{x}{y}^k.
$$
To avoid confusion, coordinatewise powers in a based permutation representation will instead be denoted by $x^{\odot k}$.

For real representations $L$ and $K$, let $\Poly_k(L,K)$ denote the space of homogeneous polynomial maps $P:L\to K$ of degree $k$. The $G$-action is $(g\cdot P)(x)=gP(g^{-1}x)$, so that
$$
\Poly_k(L,K)^G
=
\{P\in\Poly_k(L,K):P(gx)=gP(x)\text{ for all }g\in G,\ x\in L\}.
$$
Similarly, $G$ acts on $\Hom(\Sym^kL,K)$ by $(g\cdot T)(y)=gT(g^{-1}y)$.

For a finite index set $I$, multi-indices are written
$\alpha\in\N_0^I$, with $|\alpha|=\sum_{i\in I}\alpha_i$.

\begin{proposition}\label{prop:poly-sym}
For every $k\ge0$,
$$
\Poly_k(L,K)^G\cong\Hom_G(\Sym^kL,K).
$$
\end{proposition}

See \cref{app:poly-sym} for a proof. The isomorphism sends $T$ to $x\mapsto T(x^k)$ and follows from polarisation. We use this identification throughout to regard homogeneous degree-$k$ equivariant polynomial maps as elements of $\Hom_G(\Sym^kL,K)$.

\subsection{Piecewise-linear maps}
Our nonlinear layers are piecewise linear, while the degree decomposition below lives in a Gaussian $L^2$ space. We first fix the class of maps to which both viewpoints apply. A map $F:L\to K$ is \emph{piecewise linear} if it is continuous and there is a finite covering of $L$ by polyhedra on each of which $F$ is affine. We write $\Hom^{\PL}_G(L,K)$ for the equivariant ones. We follow the affine-on-pieces convention of \cite{GTW2025}, which includes constant maps. In particular, applying a continuous piecewise-linear scalar function with finitely many knots coordinatewise again gives a piecewise-linear map.

\subsection{Gaussian spaces and polynomial chaos}\label{sec:chaos}

Let $L$ be an orthogonal representation of dimension $m$ and let $\gamma_L$ denote the standard Gaussian measure on $L$. This measure is invariant under $O(L)$, hence under $G$, and has moments of every order. Its Hermite expansion provides the orthogonal grading used below.

For a representation $K$, let $L^2(\gamma_L;K)$ be the Hilbert space of square-integrable maps $L\to K$ modulo null sets, with $G$ acting unitarily by $(g\cdot F)(x)=gF(g^{-1}x)$.

We use the probabilists' Hermite polynomials $\He_k$, defined by the generating function $\exp(tx-t^2/2)=\sum_{k\ge0}\He_k(x)t^k/k!$ and normalised so that $\E[\He_j(Z)\He_k(Z)]=k!\,\delta_{jk}$ for $Z\sim N(0,1)$.

In $m$ variables, the products $\prod_i\He_{k_i}(x_i)$ with $\sum_ik_i=k$ span the $k$th \emph{polynomial chaos} $\cH_k(L)$, and
\begin{equation}\label{eq:chaos}
L^2(\gamma_L)=\whbigoplus_{k\ge0}\cH_k(L).
\end{equation}
Thus the chaos index is a polynomial degree after orthogonal removal of lower-degree terms. For example, $\He_2(x)=x^2-1$ lies entirely in the second chaos. We write $\Pi_k$ for the orthogonal projection onto $\cH_k(L)$ and use the same symbol for the induced projection on $K$-valued maps:
$$\Pi_k:L^2(\gamma_L;K)\longrightarrow \cH_k(L)\otimes K.$$

The chaos projection $\Pi_k$ restricts to an $O(L)$-equivariant isomorphism from homogeneous degree-$k$ polynomials onto $\cH_k(L)$, and hence realises the identification
$$
\cH_k(L)\cong\Sym^k(L^*)\cong\Sym^kL,
$$
using the fixed invariant inner product (see \cite{Janson1997}). For a linear form $\ell(x)=\langle x,a\rangle$ with $\|a\|=s>0$, $\Pi_k(\ell^k) = s^k\He_k(\ell/s)$.

If $P$ is homogeneous of degree $k$, then
$P-\Pi_kP$ has degree strictly less than $k$. The space $\cH_k(L)$ therefore
consists of degree-at-most-$k$ polynomials orthogonal to all
lower-degree polynomials, rather than homogeneous polynomials themselves.

The following identity will be used in the factorisation and norm calculations.

\begin{lemma}[Mehler]\label{lem:mehler}
Let $(X,Y)$ be jointly standard normal with correlation $\rho$. Then
$$
\E[\He_a(X)\He_b(Y)]
=
a!\,\rho^a\,\delta_{ab}.
$$
\end{lemma}
\begin{proof}
    This is a standard identity. See \cite[Theorem~3.9]{Janson1997}.
\end{proof}

To use this Hilbert-space decomposition for the piecewise-linear maps above, we only need to check square-integrability.

\begin{proposition}\label{prop:integrable}
Let $L$ and $K$ be finite-dimensional. If $F:L\to K$ is piecewise linear with finitely many polyhedral pieces, then $F\in L^2(\gamma_L;K)$.
\end{proposition}

\begin{proof}
Since $F$ is affine on finitely many polyhedral pieces, there exist constants $C,D\geq 0$ such that $\|F(x)\|\leq C\|x\|+D$ for all $x\in L$. Hence $\|F(x)\|^2\leq 2C^2\|x\|^2+2D^2$. If $X\sim\gamma_L$, then $\E\|X\|^2<\infty$, so
$$
\int_L \|F(x)\|^2\,d\gamma_L(x)
\leq
2C^2\E\|X\|^2+2D^2<\infty.
$$
Thus $F\in L^2(\gamma_L;K)$.
\end{proof}

Two continuous maps agreeing $\gamma_L$-almost everywhere agree everywhere, because $\gamma_L$ has full support and the set where they differ is open. Consequently, the chaos components determine a continuous map uniquely.

\vspace{0.7cm}
\section{Degree-resolved equivariant maps}\label{sec:gaussian}

We first identify the equivariant part of each polynomial chaos. We then extend the grading to several source summands and relate occurrence at some degree to the finite-group existence criterion.

\subsection{The equivariant decomposition}
The Gaussian grading is useful here only if it respects equivariance. The next theorem says that it does so, degree by degree.

\begin{theorem}\label{thm:gaussian}
Let $G$ be a finite or compact group acting orthogonally on finite-dimensional real spaces $L$ and $K$. There is a canonical isomorphism of Hilbert spaces
$$
L^2(\gamma_L;K)^G\cong\whbigoplus_{k\ge0}\Hom_G(\Sym^kL,K).
$$
\end{theorem}
\begin{proof}
Tensoring \eqref{eq:chaos} with the finite-dimensional space $K$ gives $L^2(\gamma_L;K)=\whbigoplus_k(\cH_k(L)\otimes K)$. Each chaos projection is $O(L)$-equivariant. Since $G\subseteq O(L)$, it is also $G$-equivariant and hence commutes with the projection onto $G$-fixed vectors. Therefore taking fixed points commutes with the orthogonal sum. Finally
$$
(\cH_k(L)\otimes K)^G\cong(\Sym^k(L^*)\otimes K)^G\cong\Hom_G(\Sym^kL,K)
$$
using the $O(L)$-equivariant identification
$\cH_k(L)\cong\Sym^kL$ above and the tensor--$\Hom$ identification.
\end{proof}

The key point is that each chaos projection is equivariant under the full orthogonal group. It therefore preserves the $G$-fixed subspace, allowing the Gaussian grading to restrict to equivariant maps.

\begin{corollary}\label{cor:pl-chaos}
Every equivariant piecewise-linear $F:L\to K$ has a unique sequence $(\Pi_kF)_{k\ge0}$ with $\Pi_kF\in\Hom_G(\Sym^kL,K)$, and $F\mapsto(\Pi_kF)_k$ is injective on continuous square-integrable equivariant maps.
\end{corollary}
\begin{proof}
Square-integrability follows from \cref{prop:integrable}. Orthogonality gives uniqueness in $L^2$, and full support of the Gaussian measure gives pointwise uniqueness for continuous maps.
\end{proof}

\Cref{cor:pl-chaos} establishes injectivity but does not characterise the image of the transform. In particular, not every square-summable sequence of equivariant polynomial components arises from a piecewise-linear map. The finite-group existence questions considered here do not require such a characterisation. The distinction becomes essential for compact connected groups, as discussed in \cref{sec:limitations}.

\subsection{Multi-sources and multidegrees}\label{sec:multichaos}
Fix a finite orthogonal $G$-stable decomposition $V=\bigoplus_{i\in I}W_i$ into nonzero subrepresentations. The summands need not be irreducible or pairwise non-isomorphic. The multidegree refinement recorded below separates the contributions of the summands $W_i$ within a single total degree. This allows the coordinatewise results of \cref{sec:coordinatewise} to be applied to a direct-sum input without treating its summands separately.

Multi-source inputs are not used in our constructions; the following decompositions indicate how the theory extends to them.

Symmetric powers of a direct sum split by multidegree. Iterating $\Sym(U\oplus W)\cong\Sym U\otimes\Sym W$ gives the $G$-equivariant isomorphism
$$
\Sym^kV\cong\bigoplus_{|\alpha|=k}\bigotimes_{i\in I}\Sym^{\alpha_i}W_i.
$$
Hence for any $K$,
$$
\Hom_G(\Sym^kV,K)\cong\bigoplus_{|\alpha|=k}\Hom_G\Bigl(\bigotimes_i\Sym^{\alpha_i}W_i,K\Bigr).
$$

\begin{example}\label{ex:two-source}
For $V=W_1\oplus W_2$,
$$
\Sym^2V\cong\Sym^2W_1\oplus(W_1\otimes W_2)\oplus\Sym^2W_2 ,
$$
$$
\Sym^3V\cong\Sym^3W_1\oplus(\Sym^2W_1\otimes W_2)\oplus(W_1\otimes\Sym^2W_2)\oplus\Sym^3W_2 .
$$
The mixed summands have no counterpart in a GTW interaction graph whose edges have a single source. Instead, they record simultaneous nonlinear dependence on two summands.
\end{example}

On the Gaussian side, orthogonality of the $W_i$ makes the components of $x\sim\gamma_V$ independent, so $\gamma_V=\bigotimes_i\gamma_{W_i}$, and the Hilbert tensor product gives the \emph{multichaos} decomposition
$$
L^2(\gamma_V)\cong\whbigoplus_{\alpha\in\N_0^I}\bigotimes_i\cH_{\alpha_i}(W_i),
$$
whose total-degree-$k$ part is the sum over $|\alpha|=k$. Write $\Pi_\alpha$ for the projection onto the $\alpha$th multichaos. Tensoring with $K$ and taking fixed points degreewise, as in \cref{thm:gaussian}, gives
$$
L^2(\gamma_V;K)^G\cong\whbigoplus_{\alpha\in\N_0^I}\Hom_G\Bigl(\bigotimes_i\Sym^{\alpha_i}W_i,K\Bigr).
$$

\subsection{Degree sets and detection}
We now forget multiplicities and record only the degrees in which the target can occur. This is the information needed for the existence and parity questions below.

\begin{definition}\label{def:degree-set}
For representations $L,K$ define the \emph{degree set} 
$$
D(L,K):=\{k\ge0:\Hom_G(\Sym^kL,K)\ne0\},
$$
and $d(L,K):=\min D(L,K)\in\N_0\cup\{\infty\}$. We use $\min\varnothing=\infty$. 
\end{definition}

If $K^G\ne0$, then degree zero records constant equivariant maps. When the question is transfer of input information, the first positive degree may therefore be the more relevant quantity.

\begin{proposition}\label{prop:degree-parity}
Let $L\neq 0$, and let $G$ be finite or compact. If $k\in D(L,K)$, then $k+2m\in D(L,K)$ for every $m\in\N_0$. Consequently, for each $\varepsilon\in\{0,1\}$ define
$$
d_\varepsilon(L,K)
:=
\min\{k\in D(L,K):k\equiv\varepsilon\pmod 2\}.
$$
Then
$$
D(L,K)
=
\bigl(d_0(L,K)+2\N_0\bigr)
\cup
\bigl(d_1(L,K)+2\N_0\bigr),
$$
where a term is omitted if $d_\varepsilon(L,K)$ is infinite.
\end{proposition}

\begin{proof}
Since $G$ is finite or compact, $L$ admits a $G$-invariant inner product, so $q(x)=\|x\|^2$ is a nonzero invariant quadratic polynomial. Hence multiplication by $q$ sends
$$
\Poly_k(L,K)^G \longrightarrow \Poly_{k+2}(L,K)^G.
$$
This map is injective because the polynomial ring on $L$ is an integral domain. Hence $k\in D(L,K)\implies k+2\in D(L,K)$.
Thus, within each parity, occurrence persists from the first occurring degree onward, giving the stated $D(L,K)$, with the corresponding term omitted when $d_\varepsilon(L,K)=\infty$.
\end{proof}

Thus, once the first even and odd occurrences are known, the entire degree set is known. The remaining task is to determine those two thresholds.

\begin{theorem}[Finite-group degree detection]\label{thm:detection}
Let $L\ne0$ and $K$ be (not necessarily irreducible) real $G$-representations, and put $N=\ker L$ and $Q=G/N$. Then
$$
D(L,K)\ne\varnothing
\quad\Longleftrightarrow\quad
K^N\ne0.
$$
In this case
$d(L,K)\le|Q|-1$. If $K$ is irreducible, the condition is equivalent to
$\ker L\subseteq\ker K$.
\end{theorem}
\begin{proof}
Suppose $k\in D(L,K)$ and choose $0\ne T\in\Hom_G(\Sym^kL,K)$. Since $N=\ker L$ acts trivially on
$\Sym^kL$, we have $nT(y)=T(ny)=T(y)$ for all $n \in N$ and $y \in \Sym^kL$. Thus $\operatorname{im}T\subseteq K^N$, and since $T\ne0$ we obtain $K^N\ne0$. If $K$ is irreducible, then $K^N$ is $G$-stable because $N\triangleleft G$. Hence $K^N\ne0$ implies $K^N=K$. Equivalently, $\ker L\subseteq\ker K$.

For the converse, assume $K^N\ne0$. Both $L$ and $K^N$ descend to $Q=G/N$, and $L$ is faithful as a $Q$-representation. We argue in three steps.

\emph{A free orbit exists.} For $1\ne q\in Q$ the fixed space $L^q$ is a proper subspace, since $L$ is faithful. The finite union of these proper subspaces cannot cover $L$, so some $v\in L$ has trivial stabiliser, and the orbit $Qv$ consists of $|Q|$ distinct points.

\emph{The orbit can be separated by a linear form.}
A functional $\ell\in L^*$ fails to separate two distinct orbit points $qv$ and $q'v$ precisely when $\ell(qv-q'v)=0$. Since $qv-q'v\ne0$, the set of such functionals is a proper hyperplane
of $L^*$. There are only finitely many pairs of orbit points, so a
finite union of these hyperplanes cannot cover $L^*$. Hence some
$\ell\in L^*$ separates all points of $Qv$.

\emph{Interpolate.} With such an $\ell$, the Lagrange polynomial
$$
p(x)=\prod_{1\ne q\in Q}\frac{\ell(x)-\ell(qv)}{\ell(v)-\ell(qv)}
$$
has degree at most $|Q|-1$, with $p(v)=1$ and $p(qv)=0$. Fix $0\ne w\in K^N$ and average:
$$
F(x)=\sum_{q\in Q}p(q^{-1}x)\,qw .
$$
Reindexing $q=aq'$ shows $F(ax)=aF(x)$, so $F$ is equivariant, and $F(v)=w\ne0$, so $F\ne0$. The action is linear, so each homogeneous component $F_j$ of $F$ is separately equivariant, and some $F_j\ne0$ with $j\le|Q|-1$. So $\Poly_j(L,K^N)^Q \neq 0$. Hence $\Hom_Q(\Sym^jL,K^N)\ne0$. Inflating along $G\twoheadrightarrow Q$ and
composing with the inclusion $K^N\hookrightarrow K$ gives
$0\ne\Hom_G(\Sym^jL,K^N)\subseteq\Hom_G(\Sym^jL,K)$, so $j\in D(L,K)$.
\end{proof}

The key point here is that every irreducible representation of a finite group occurs somewhere in the symmetric algebra of a faithful representation. The interpolation proof gives the explicit bound $d(L,K)\le|G/\ker L|-1$. No such argument is available for a connected group, which is consistent with the failure discussed in \cref{sec:limitations}.

For irreducible $L,K$, combining \cref{thm:detection} with \eqref{eq:gtw-intro} identifies the GTW existence criterion with $D(L,K)\ne\varnothing$. The following refinement also decides whether each parity occurs.

\begin{corollary}\label{cor:parity-existence}
Let $L\neq0$ and let $K$ be irreducible with
$\ker L\subseteq\ker K$. Put $Q=G/\ker L$, so that $Q$ acts
faithfully on $L$ and naturally on $K$.

\begin{enumerate}
    \item If no element of $Q$ acts as $-I_L$, then both even and odd
    degrees occur, with
    $$
    d_0(L,K),\,d_1(L,K)\le 2|Q|-1.
    $$

    \item If some $z\in Q$ acts as $-I_L$, then $z$ is a central
    involution and acts on $K$ as either $I_K$ or $-I_K$.
    In the first case only even degrees occur, and in the second only
    odd degrees occur. In either case, the occurring parity has threshold
    at most $|Q|-1$.
\end{enumerate}
\end{corollary}

\begin{proof}
Fix $\varepsilon\in\{0,1\}$ and let
$\widetilde Q=Q\times\{\pm1\}$ act on $L$ and $K$ by
$$
(q,t)\cdot x=t\,qx,
\qquad
(q,t)\cdot y=t^\varepsilon qy.
$$
A homogeneous $Q$-equivariant polynomial $P:L\to K$ is
$\widetilde Q$-equivariant exactly when its degree is congruent to
$\varepsilon$ modulo $2$.

Suppose first that no element of $Q$ acts as $-I_L$. Then the
$\widetilde Q$-action on $L$ is faithful. By \cref{thm:detection},
there exists a nonzero $\widetilde Q$-equivariant homogeneous polynomial
of degree at most $|\widetilde Q|-1=2|Q|-1$. Applying this for $\varepsilon=0$ and $\varepsilon=1$ gives both parities.

Now suppose $z\in Q$ acts as $-I_L$. Since $Q$ acts faithfully on $L$,
we have $z^2=1$, and for every $q\in Q$ the elements $qzq^{-1}$ and $z$
act identically on $L$. Hence $z$ is central. The kernel of the $\widetilde Q$-action on $L$ is therefore $\{(1,1),(z,-1)\}$.

By \cref{thm:detection}, a degree of parity $\varepsilon$ occurs exactly
when this kernel acts trivially on $K$, that is, when
$$
(-1)^\varepsilon zy=y
\qquad\text{for all }y\in K.
$$

Since $z$ is central, its action on $K$ lies in $\End_G(K)$. Moreover
$z^2=1$, so irreducibility implies that $z$ acts as either $I_K$ or
$-I_K$. Hence exactly one parity occurs: even if $z$ acts as $I_K$, and
odd if it acts as $-I_K$. The faithful quotient has order $|Q|$, so the
corresponding threshold is at most $|Q|-1$.
\end{proof}

\subsection{Computing degree sets}
The detection theorem decides whether some degree occurs, but not where the first occurrences lie. Molien's formula packages the dimensions in all degrees into one generating series.

\begin{proposition}[Molien series]\label{prop:molien}
Let $G$ be finite and $c_k=\dim\Hom_G(\Sym^kL,K)$. Then
$$
\sum_{k\ge0}c_kt^k=\frac1{|G|}\sum_{g\in G}\frac{\chi_K(g)}{\det(I-t\rho_L(g))}
=\frac1{|G|}\sum_{g\in G}\chi_K(g)\exp\Bigl(\sum_{m\ge1}\frac{t^m}{m}\chi_L(g^m)\Bigr).
$$
Here $\rho_L:G\to\mathrm{GL}(L)$ is the homomorphism defining the action of $G$ on $L$, and $\chi_L$ and $\chi_K$ are the characters of $L$ and $K$.
\end{proposition}

This is the usual character average; see \cref{app:molien}. The exponential form is useful computationally because it uses only the power-map characters $\chi_L(g^m)$, which are fixed-point counts for permutation representations. We use this form for subset representations in \cref{thm:molien-polya}.

When $K$ is irreducible, the coefficient $c_k$ need not equal the multiplicity of $K$ in $\Sym^kL$ over $\R$, but rather
$$
c_k = [\Sym^kL:K]\, \dim_\R\End_G(K).
$$

In small representations one can often read the degree set directly from the symmetric powers. The next two examples illustrate \cref{thm:detection,prop:degree-parity}.

\begin{example}[One-dimensional sources]\label{ex:one-dim}
Suppose $\dim L=1$. Then $\Sym^kL=L^{\otimes k}$.
If $\chi$ is the real character of $L$, its order is
$q\in\{1,2\}$. For the one-dimensional target $\chi^a$,
$a\in\{0,1\}$,
$$
D(L,\chi^a) = \{k\ge0:k\equiv a\pmod q\}.
$$
Thus a trivial source reaches the trivial target in every degree, while a source of order two reaches the trivial target in even degrees and itself in odd degrees.
\end{example}

\begin{example}[Rotation representations of $C_n$]\label{ex:cyclic}
Let $C_n=\langle a\rangle$, and let $\rho_s$ and $\rho_t$
be irreducible two-dimensional real rotation representations, with
$2s,2t\not\equiv0\pmod n$. Writing $(\rho_s)_\C\cong W_s\oplus W_{-s}$,
define
$$
R_k=\{r\in\Z:|r|\le k,\ r\equiv k\pmod2\}.
$$
Then $\Sym^k((\rho_s)_\C) \cong \bigoplus_{r\in R_k}W_{rs}$,
and therefore
$$
D(\rho_s,\rho_t) = \{k\ge0:\exists r\in R_k \text{ with }rs\equiv\pm t\pmod n\}. 
$$
This set is nonempty exactly when $\gcd(s,n)\mid t$, or 
equivalently when $\ker\rho_s\subseteq\ker\rho_t$. Since $R_k\subseteq R_{k+2}$, occurrence persists within each parity, as predicted by \cref{prop:degree-parity}.
\end{example}

\vspace{0.7cm}
\section{Coordinatewise nonlinearities}\label{sec:coordinatewise}

The preceding sections determine which equivariant polynomial maps
exist. We now ask which of these maps a coordinatewise layer produces.

The input representation $L$ is arbitrary throughout and need not
be irreducible or multiplicity-free. In particular, the results
include interactions between components of a direct-sum input,
without requiring a separate treatment of those components.

\subsection{Coordinatewise layers and their degree components}
\label{sec:act-setting}
There are two pieces of data to track: the polynomial produced by the hidden coordinates, and the scalar Hermite coefficient supplied by the activation. We introduce them separately.

Let $L$ and $K$ be finite-dimensional real orthogonal
$G$-representations, regarded as the input and output spaces. Coordinatewise application is equivariant only after choosing coordinates that the group permutes. A \emph{hidden representation} is an intermediate representation $H=\R^{\cB}$ with an orthonormal basis $\cB$ permuted by $G$. Throughout this
section we assume that the action on $\cB$ is transitive.

No inclusion relations between $L$, $H$, and $K$ are assumed. They are connected by equivariant linear maps
$$ 
A:L\longrightarrow H, \qquad R:H\longrightarrow K.
$$
Thus a coordinatewise layer has the form
$$
L\xrightarrow{A}H
\xrightarrow{\sigma_{\cB}}H
\xrightarrow{R}K,
$$
where the nonlinearity is applied in the distinguished coordinate basis
of $H$.

All adjoints are taken with respect to the fixed inner products.
Hidden representations with several coordinate orbits are treated in
\cref{app:intransitive}.

\begin{definition}
For an activation $\sigma:\R\to\R$, its
\emph{coordinatewise extension} to $H$ is
$$
\sigma_{\cB}\left(\sum_{b\in\cB}z_bb\right)
=
\sum_{b\in\cB}\sigma(z_b)b.
$$
The associated \emph{coordinatewise layer} is
$$
F:L\longrightarrow K,
\qquad
F(x)=R\sigma_{\cB}(Ax).
$$
\end{definition}

Since $\sigma_{\cB}$ commutes with permutations of $\cB$,
the layer $F$ is equivariant. We assume $A\ne0$.

\begin{definition}
For $b\in\cB$, define the coordinate vector $a_b=A^*b\in L$. Thus the $b$th hidden coordinate is
$(Ax)_b=\langle x,a_b\rangle$.
The \emph{coordinate Gram matrix} is
$$
C=(C_{bb'})_{b,b'\in\cB},
\qquad
C_{bb'}=\langle a_b,a_{b'}\rangle.
$$
Thus $C$ is the matrix of $AA^*$ in the basis $\cB$.
\end{definition}

\begin{lemma}
There is an $s>0$ such that $\norm{a_b}=s$ for every $b\in\cB$.
Moreover,
$$
s^2=C_{bb}=\frac{\tr C}{\dim H}.
$$
For $X\sim\gamma_L$, every coordinate $(AX)_b$ therefore
has distribution $N(0,s^2)$.
\end{lemma}

\begin{proof}
Equivariance and orthogonality give $a_{gb}=ga_b$.
Transitivity implies that all coordinate vectors have the
same norm, and $A\ne0$ implies that this norm is positive.
Summing the diagonal entries of $C$ gives the trace identity.
Finally, $\langle X,a_b\rangle$ is centred Gaussian with
variance $\norm{a_b}^2$.
\end{proof}

The activation enters the $k$th chaos through a single scalar coefficient.

\begin{definition}
Suppose $\sigma(sZ)$ is square-integrable for $Z\sim N(0,1)$.
Its \emph{Hermite coefficients at scale $s$} are
$$
a_k(\sigma;s)
=
\frac1{k!}\E[\sigma(sZ)\He_k(Z)],
\qquad
a_k(\sigma):=a_k(\sigma;1).
$$
\end{definition}

We impose this integrability assumption throughout.
It holds for every continuous piecewise-linear activation
with finitely many knots, since such a function has at most
linear growth.

The remaining degree-$k$ information comes from the fixed lift and readout. It is encoded by the coordinatewise $k$th power.

\begin{definition}
For $z=\sum_bz_bb\in H$, define
$$
z^{\odot k}:=\sum_bz_b^k b\quad(k\ge0).
$$
The \emph{degree-$k$ coordinate polynomial} associated with
$A$ and $R$ is
$$
Q_k:L\longrightarrow K,
\qquad
Q_k(x)
=
R\bigl((Ax)^{\odot k}\bigr)
=
\sum_{b\in\cB}\langle x,a_b\rangle^k Rb.
$$
For matrices, $\circ$ denotes entrywise multiplication,
and $C^{\circ k}$ denotes the entrywise power, with
$C^{\circ0}=J$, the all-ones matrix.
\end{definition}

Coordinatewise powers commute with permutations, so
$Q_k\in\Poly_k(L,K)^G$. On homogeneous degree-$k$ polynomials, $\Pi_k$ is an isomorphism onto $\cH_k(L)$. In particular,
$$
\Pi_kQ_k=0
\quad\Longleftrightarrow\quad
Q_k=0.
$$
All $L^2$ norms below use the standard Gaussian on the
source and the fixed inner product on the target.

With these definitions, the activation and the coordinate geometry separate as follows.

\begin{theorem}[Coordinatewise factorisation and Gram formula]
\label{thm:master}
Under the preceding assumptions, for every $k\ge0$,
\begin{align}
\Pi_kF
&=
\frac{a_k(\sigma;s)}{s^k}\,\Pi_kQ_k,
\label{eq:master}\\[4pt]
\norm{\Pi_kF}_{L^2}^2
&=
\frac{k!\,|a_k(\sigma;s)|^2}{s^{2k}}
\tr\bigl(RC^{\circ k}R^*\bigr).
\notag
\end{align}
Consequently,
$$
\begin{aligned}
\Pi_kF\ne0
&\quad\Longleftrightarrow\quad
a_k(\sigma;s)\ne0
\ \text{and}\ Q_k\ne0,\\
Q_k\ne0
&\quad\Longleftrightarrow\quad
\tr(RC^{\circ k}R^*)>0.
\end{aligned}
$$
No irreducibility or multiplicity-free assumption is required.
\end{theorem}

The transitivity assumption is used only to give all hidden coordinates the same variance. The corresponding formula for several coordinate orbits is recorded in \cref{app:intransitive}.

\begin{proof}
For each $b$, the random variable $\langle X,a_b\rangle/s$
is standard normal. Hence
$$
\sigma(\langle x,a_b\rangle)
=
\sum_{k\ge0}
a_k(\sigma;s)\He_k(\langle x,a_b\rangle/s)
$$
in $L^2(\gamma_L)$. The $k$th summand belongs to $\cH_k(L)$,
and
$$
\Pi_k\bigl(\langle x,a_b\rangle^k\bigr)
=
s^k\He_k(\langle x,a_b\rangle/s).
$$
Multiplying by $Rb$ and summing over $b$ proves
\eqref{eq:master}.

Applying \cref{lem:mehler} to the hidden coordinates gives
$$
\left\langle
\Pi_k\bigl(\langle x,a_b\rangle^k\bigr),
\Pi_k\bigl(\langle x,a_{b'}\rangle^k\bigr)
\right\rangle_{L^2}
=
k!C_{bb'}^k.
$$
Therefore
$$
\norm{\Pi_kQ_k}_{L^2}^2
=
k!\sum_{b,b'}C_{bb'}^k\langle Rb,Rb'\rangle
=
k!\tr(RC^{\circ k}R^*).
$$
Combining this identity with \eqref{eq:master} proves
the norm formula. The nonvanishing statements follow
from injectivity of $\Pi_k$ on homogeneous degree-$k$
polynomials.
\end{proof}
The theorem separates the three questions~\ref{question:symmetry}--\ref{question:activation} of the introduction. For fixed $A$ and $R$, changing the activation only changes the scalar multiplying $\Pi_kQ_k$; it cannot select a different polynomial direction at degree $k$.

\begin{example}
\label{ex:S3-relu}
Let $H=\R^3$ be the $S_3$ permutation representation, and recall the trivial and standard subrepresentations, $L_0$ and $L_1$.

Let $A:L_1\hookrightarrow H$ be inclusion, and
$R=p_{L_0}$ be orthogonal projection. Then $C=p_{L_1}$ and
$s^2=\dim L_1/3=2/3$. As in \eqref{eq:intro-square},
$$
Q_2(x)=\frac13\norm{x}^2\one,
\qquad
\Pi_2Q_2(x)=\frac13(\norm{x}^2-2)\one.
$$
Here $\E[\norm{X}^2]=2$ for $X\sim\gamma_{L_1}$.
Thus
$$
\Pi_2F(x)
=
\frac{a_2(\sigma;\sqrt{2/3})}{2}
(\norm{x}^2-2)\one.
$$
The coordinate polynomial is nonzero, so this second-chaos
interaction is present precisely when
$a_2(\sigma;\sqrt{2/3})\ne0$.
\end{example}

\subsection{The submodule generated by the coordinates}

We next explain why the polynomial $Q_k$ can vanish even
when $\Hom_G(\Sym^kL,K)$ is nonzero.

\begin{definition}
For $k\ge0$, define
$$
Y_k
=
\Span\{a_b^k:b\in\cB\}
\subseteq\Sym^kL.
$$
We call $Y_k$ the \emph{coordinate-generated submodule}
of degree $k$. Define also
$$
\phi_k:H\longrightarrow\Sym^kL,
\qquad
\phi_k(b)=a_b^k.
$$
\end{definition}

The relation $a_{gb}=ga_b$ implies that $\phi_k$ is equivariant.
Thus $Y_k=\operatorname{im}\phi_k$ is indeed a submodule.

\begin{proposition}
\label{prop:coordinate-image}
For every $k\ge0$, $\phi_k^*\phi_k=C^{\circ k}$, and so
$$
Y_k\cong(\ker C^{\circ k})^\perp
$$
as $G$-representations. \Cref{prop:poly-sym} gives $Q_k(x)=R\phi_k^*(x^k)$. Thus
$$
Q_k\ne0
\quad\Longleftrightarrow\quad
R\phi_k^*\ne0.
$$
\end{proposition}

\begin{proof}
The tensor inner product gives
$$
\langle a_b^k,a_{b'}^k\rangle
=
\langle a_b,a_{b'}\rangle^k
=
C_{bb'}^k.
$$
Hence $C^{\circ k}$ is the Gram matrix of the vectors
$a_b^k$, proving
$\phi_k^*\phi_k=C^{\circ k}$. In particular,
$\ker\phi_k=\ker C^{\circ k}$. Restricting $\phi_k$
to the orthogonal complement of its kernel gives
the stated isomorphism onto $Y_k$.

For $z\in\Sym^kL$,
$$
R\phi_k^*z
=
\sum_b\langle z,a_b^k\rangle Rb.
$$
Taking $z=x^k$ gives $Q_k(x)$.
Since pure powers span $\Sym^kL$, the polynomial vanishes
identically if and only if $R\phi_k^*$ is zero.
\end{proof}

An arbitrary degree-$k$ equivariant polynomial may use all of $\Sym^kL$, whereas the coordinate construction sees only $Y_k$. A target can therefore occur in $\Sym^kL$ and still be invisible to the chosen coordinates. Even when it occurs in $Y_k$, the readout $R$ may kill the resulting direction.

\subsection{Projection onto irreducible summands}
\label{sec:krein}
When the lift is an inclusion and the readout is an orthogonal projection, the Gram criterion becomes especially transparent. Let $W,K\subseteq H$ be nonzero subrepresentations and take
$$
A:W\hookrightarrow H,
\qquad
R=p_K:H\longrightarrow K.
$$
Writing $E_W$ and $E_K$ for the orthogonal projector
matrices in $\cB$, we have
$$
C=E_W,
\qquad
s^2=\frac{\dim W}{\dim H},
\qquad
Q_k(x)=p_K(x^{\odot k})\quad(x\in W).
$$
The norm formula becomes
$$
\norm{\Pi_kQ_k}_{L^2(\gamma_W;K)}^2
=
k!\,\tr(E_KE_W^{\circ k}),
$$
which holds without a multiplicity-free assumption.

Suppose now that $H$ is multiplicity-free,
$H=\bigoplus_{j=0}^mL_j$ with the $L_j$ pairwise non-isomorphic and irreducible.
Let $L_0$ be the trivial summand, let $E_j$ be orthogonal
projection onto $L_j$, and let $d_j=\dim L_j$.
Transitivity gives
$$
E_0=\frac{J}{\dim H},
\qquad
(E_j)_{bb}=\frac{d_j}{\dim H}.
$$

When $H$ is multiplicity-free, the matrix $E_W^{\circ k}$ is determined by one scalar on each irreducible summand. We record these scalars next.

\begin{definition}
For the fixed source subrepresentation $W$ (not necessarily irreducible), define
$$
\kappa_j(k)
:=
\frac1{d_j}\tr(E_jE_W^{\circ k}),
\qquad
k\ge0.
$$
We call $\kappa_j(k)$ the \emph{projector coefficient}
for the source $W$, target $L_j$ and degree $k$.
\end{definition}

\begin{corollary}
\label{thm:krein}
For each $k\ge0$:

\begin{enumerate}
    \item The Hadamard power of the projector decomposes as
    $$
    E_W^{\circ k}
    =
    \sum_{j=0}^m \kappa_j(k)E_j,
    \qquad
    \kappa_j(k)\ge0.
    $$

    \item For $Q_k(x)=p_j(x^{\odot k})$ with $x\in W$,
    $$
    Q_k\ne0
    \quad\Longleftrightarrow\quad
    \kappa_j(k)>0, \qquad
    \norm{\Pi_kQ_k}_{L^2}^2
    =
    k!\,d_j\,\kappa_j(k).
    $$

    \item The coordinate-generated submodule for the inclusion
    $W\hookrightarrow H$ is
    $$
    Y_k
    \cong
    \bigoplus_{j:\,\kappa_j(k)>0}L_j.
    $$
\end{enumerate}
\end{corollary}

\begin{proof}
Hadamard products of $G$-equivariant matrices are again equivariant. Thus $E_W^{\circ k}$ is a self-adjoint $G$-endomorphism of $H=\bigoplus_{j=0}^m L_j$.
By multiplicity-freeness and Schur's lemma, together with
self-adjointness, it acts on each $L_j$ by a real scalar. Hence $E_W^{\circ k} = \sum_{j=0}^m \lambda_j(k)E_j$. Multiplying by $E_j$ and taking traces proves $E_W^{\circ k} = \sum_{j=0}^m\kappa_j(k)E_j$.

Since $E_W$ is positive semidefinite, the Schur product theorem implies
that $E_W^{\circ k}$ is positive semidefinite. Its scalar on $L_j$ is
therefore nonnegative, so $\kappa_j(k)\ge0$.

For the inclusion $W\hookrightarrow H$ and projection
$p_j:H\to L_j$, the norm formula gives
$$
\norm{\Pi_kQ_k}_{L^2}^2
=
k!\,\tr(E_jE_W^{\circ k})
=
k!\,d_j\kappa_j(k).
$$
Since $\Pi_k$ is injective on homogeneous degree-$k$ polynomials,
$$
Q_k\ne0
\quad\Longleftrightarrow\quad
\kappa_j(k)>0.
$$

Finally, \cref{prop:coordinate-image} gives $Y_k\cong(\ker E_W^{\circ k})^\perp$. Since $E_W^{\circ k}$ acts on $L_j$ by $\kappa_j(k)$, its kernel is
the sum of those $L_j$ with $\kappa_j(k)=0$. Hence
$$
Y_k
\cong
\bigoplus_{j:\,\kappa_j(k)>0}L_j.
$$
\end{proof}

When $W$ is an irreducible summand and $k=2$, these Hadamard coefficients are the usual Krein parameters, the structure constants for entrywise products of primitive idempotents, up to the chosen normalisation. Higher $k$ gives iterated Hadamard products. See also \cref{app:subset-representations} for the Johnson-scheme case.

\begin{lemma}
\label{lem:parity}
For each $k \geq 0$,
$$\kappa_j(k)>0 \implies \kappa_j(k+2)>0.$$

Equivalently, if $p_j(x^{\odot k})|_W$ is nonzero, then
$p_j(x^{\odot(k+2)})|_W$ is nonzero. Hence, for each $j$, the supported degrees form at most one even tail and one odd tail.
\end{lemma}

\begin{proof}
For the trivial summand $L_0$,
$$
\kappa_0(2)
=
\tr(E_0E_W^{\circ2})
=
\frac{\dim W}{\dim H}.
$$
Using the decomposition from \cref{thm:krein},
$$
E_W^{\circ2}
=
\frac{\dim W}{\dim H}E_0+B
=
\frac{\dim W}{(\dim H)^2}J+B,
$$
where $B$ is positive semidefinite. Therefore
$$
E_W^{\circ(k+2)}
=
E_W^{\circ k}\circ E_W^{\circ2}
=
\frac{\dim W}{(\dim H)^2}E_W^{\circ k}
+
E_W^{\circ k}\circ B.
$$
The second term is positive semidefinite by the Schur product theorem.
Its projector coefficients are therefore nonnegative, so
$$
\kappa_j(k+2)
\ge
\frac{\dim W}{(\dim H)^2}\kappa_j(k).
$$
Since $W \neq 0$, positivity of $\kappa_j(k)$ implies positivity of $\kappa_j(k+2)$. The equivalence with the coordinate polynomial follows from \cref{thm:krein}.
\end{proof}

This is the coordinate-level analogue of \cref{prop:degree-parity}: within each parity, coordinate support also persists once it appears.

\subsection{Examples of realisation and obstruction}

The first example shows that a permitted quadratic map
can be missed by the coordinate construction.

\begin{example}
\label{ex:strict-coordinate}
Let $S_4$ act on the six unordered pairs of $[4]:=\{1,2,3,4\}$, and let
$H=\R^{\binom{[4]}2}$ be the corresponding permutation
representation. Its standard subrepresentation is
$$
W=
\left\{
x_{ab}=u_a+u_b:
u\in\R^4,\ \sum_a u_a=0
\right\}.
$$
The map $u\mapsto x$ is equivariant and injective, since $u_a=\frac12\sum_{b\ne a}x_{ab}
$ on its image. Thus $W$ is a copy of the Specht module $S^{(3,1)}$. Let $J_c$ be the complementation operator on edge coordinates,
defined by
$$
J_c e_A=e_{[4]\setminus A}.
$$
For $x\in W$, the zero-sum condition gives $x_{A^c}=-x_A$.
Thus $J_c$ acts as $-1$ on $W$, whereas
$$
(x^{\odot2})_{A^c}=(x^{\odot2})_A.
$$
Since $J_c$ is a self-adjoint involution, its two eigenspaces
are orthogonal. Consequently,
$$
p_W(x^{\odot2})=0
\qquad(x\in W).
$$

Nevertheless, the map
$$
u\longmapsto
\left(u_a^2-\frac14\sum_bu_b^2\right)_{a=1}^4
$$
is a nonzero equivariant quadratic map from the standard
vertex representation to itself. Transporting it through
the identification above gives $\Hom_{S_4}(\Sym^2W,W)\ne0$.

Thus symmetry permits a quadratic interaction that the
fixed layer $F(x)=p_W\sigma_{\cB}(x)$ cannot produce
in its second chaos. This holds for every activation
satisfying the integrability assumption, including biased
ReLU, because its coordinate polynomial $Q_2$ is zero.
The same argument applies in every even degree.
\end{example}

Equivalently, the Gram criterion in \cref{thm:master} gives $\tr(E_WE_W^{\circ2})=0$ in this example.

For regular representations of finite abelian groups,
every permitted degree has a nonzero coordinate polynomial.

\begin{proposition}
\label{prop:abelian-realisation}
Let $G$ be finite abelian, and let $H=\R G=\bigoplus_jL_j$ be its real regular representation, equipped with the
regular coordinate basis. Let $W\subseteq H$ be any
nonzero subrepresentation, and let $p_j$ be orthogonal
projection onto $L_j$. Then, for every $k\ge0$,
$$
p_j(x^{\odot k})|_W\not\equiv0
\quad\Longleftrightarrow\quad
\Hom_G(\Sym^kW,L_j)\ne0.
$$
\end{proposition}

\begin{proof}
After complexification, identify $H_\C$ with functions
$G\to\C$. Since $G$ is abelian, its regular representation
is a direct sum of one-dimensional character spaces.
The subrepresentation $W_\C$ is therefore spanned by a
subset of the character functions.

Consider the equivariant multiplication map
$$
\mu_k:\Sym^kW_\C\longrightarrow H_\C,
\qquad
f_1\cdots f_k\longmapsto
f_1\odot\cdots\odot f_k.
$$
A monomial in character functions is sent to their
pointwise product, which is a nonzero character function
of the same representation type. Thus every character
occurring in $\Sym^kW_\C$ occurs in
$\operatorname{im}\mu_k$, and conversely.

The polynomial $p_j(x^{\odot k})$ corresponds to
$p_j\mu_k$. Since pure powers span the symmetric power,
this polynomial is nonzero exactly when a character
of $(L_j)_\C$ occurs in $\Sym^kW_\C$.
That is, $\Hom_G(\Sym^kW,L_j)\ne0$.
\end{proof}

The proposition detects nonzero degrees only; it does not say that the selected polynomial spans the full equivariant polynomial space. For these
regular abelian representations, the presence of a
permitted degree in the layer is therefore decided
by the activation coefficient in \cref{thm:master}.

\vspace{0.7cm}
\section{Activation spectra}\label{sec:activation-spectra}

The coordinatewise factorisation separates the polynomial selected
by a layer from the coefficient supplied by its activation.
We now determine
when $a_k(\sigma;s)$ vanishes.

\begin{definition}
Let $s>0$, and suppose $\sigma(sZ)$ is square-integrable
for $Z\sim N(0,1)$. The \emph{Hermite spectrum of $\sigma$
at scale $s$} is the sequence
$(a_k(\sigma;s))_{k\ge0}$. Its \emph{support} is
$$
\supp a_\bullet(\sigma;s)
=
\{k\in\N_0:a_k(\sigma;s)\ne0\}.
$$
At scale one, write
$\supp a_\bullet(\sigma)=\supp a_\bullet(\sigma;1)$.
\end{definition}

For a layer satisfying the assumptions of \cref{thm:master},
$$
\{k:\Pi_kF\ne0\}
=
\{k:Q_k\ne0\}\cap\supp a_\bullet(\sigma;s).
$$
Thus the activation support determines which of the nonzero
coordinate polynomials contribute to the layer.

\subsection{Positively homogeneous activations}
\label{sec:homogeneous}
We start with positively homogeneous activations, which include the identity, absolute value, and ReLU.

\begin{definition}
A function $\sigma:\R\to\R$ is \emph{positively homogeneous of degree one} if
$$
\sigma(tx)=t\sigma(x)
\qquad (t\ge0,\ x\in\R).
$$
Writing $m_+=\sigma(1)$ and $m_-=-\sigma(-1)$, such a function is necessarily of the form
$$
\sigma(x)
=
\begin{cases}
m_+x,&x\ge0,\\
m_-x,&x\le0,
\end{cases}
\quad=\quad
\frac{m_++m_-}{2}\,x
+
\frac{m_+-m_-}{2}\,\abs{x}.
$$
Every such function is continuous and piecewise linear. 
\end{definition}

From here on, \emph{homogeneous} means positively homogeneous of degree one.

\begin{theorem}
\label{thm:classification}
Let $\sigma\ne0$ be homogeneous,
with slopes $m_+$ and $m_-$ as above. For every $s>0$,
$$
a_k(\sigma;s)=s\,a_k(\sigma),
$$
and its support is given by
$$
\begin{array}{lcl}
m_+=m_-&\Longrightarrow&\supp a_\bullet(\sigma)=\{1\},\\[2pt]
m_+=-m_-&\Longrightarrow&\supp a_\bullet(\sigma)=2\N_0,\\[2pt]
m_+\ne\pm m_-&\Longrightarrow&\supp a_\bullet(\sigma)=\{1\}\cup2\N_0 .
\end{array}
$$
\end{theorem}

\begin{proof}
Positive homogeneity gives $\sigma(sZ)=s\sigma(Z)$,
which proves the scaling identity.

The odd part of $\sigma$ is a multiple of $x=\He_1(x)$, which contributes only in degree one.
The function $\abs{x}$ is even, so its odd Hermite
coefficients vanish. Its degree-zero coefficient is
$\E[\abs{Z}]=2\varphi(0)>0$, where $\varphi(t)=(2\pi)^{-1/2}e^{-t^2/2}$ is the standard one-dimensional Gaussian density. For $r\ge1$, two integrations
by parts give
$$
a_{2r}(\abs{\cdot})
=
\frac{2}{(2r)!}\He_{2r-2}(0)\varphi(0).
$$
These coefficients are nonzero because
$$
\He_{2m}(0)
=
(-1)^m\frac{(2m)!}{2^m m!}.
$$
Thus $\abs{x}$ has support exactly $2\N_0$.
The two parts of $\sigma$ have disjoint supports, so
their nonzero coefficients cannot cancel.
\end{proof}

\begin{example}
    The identity activation has support $\{1\}$, absolute value has support $2\N_0$, and $\operatorname{ReLU}(x) := \max(x,0)=\frac{x+\abs{x}}2$ has support $\{1\}\cup2\N_0$.
\end{example}

The scaling identity shows that changing the input scale cannot create a degree that is absent at scale one. In particular, no odd degree $k>1$ occurs, for any $\sigma$ in the class and any input scale.

\begin{corollary}
\label{cor:odd-blind}
Let $F(x)=R\sigma_{\cB}(Ax)$ satisfy the assumptions of
\cref{thm:master}, with $\sigma$ homogeneous. Then
$$
\Pi_kF=0
\qquad\text{for every odd }k\ge3.
$$
The same conclusion holds for finite sums of such layers.
\end{corollary}

\begin{proof}
Write $\sigma(z)=\lambda z+\mu\abs z$ with $\lambda=\tfrac{m_++m_-}2$ and
$\mu=\tfrac{m_+-m_-}2$. The second summand is even, so the odd part of
the layer is
$$
\frac{F(x)-F(-x)}2
=
\lambda\,RAx ,
$$
which is linear and therefore lies in the first chaos, while an even map
has no odd chaos components at all. Hence $\Pi_kF=0$ for every odd
$k\ge3$, and the same argument applies to a finite sum of such layers.
\end{proof}

This statement applies to a single layer, or a sum of such layers, but it
does not imply the same restriction for their compositions.

\subsection{Bias and input scale}\label{sec:bias-scale}
We now show how a
shared additive bias changes this conclusion.

\begin{definition}
For a threshold $c\in\R$, define
$$
\sigma_c(x):= \operatorname{ReLU}(x-c).
$$
For fixed equivariant maps $A:L\to H$ and $R:H\to K$,
write $F_c(x)=R(\sigma_c)_{\cB}(Ax)$.
\end{definition}

The same threshold is applied to every hidden coordinate,
so $F_c$ remains equivariant. The identity $\sigma_c(sz)=s\,\sigma_{c/s}(z)$
gives $a_k(\sigma_c;s)=s\,a_k(\sigma_{c/s})$.
Thus the support depends on the ratio $c/s$. The following formula determines it explicitly.

\begin{corollary}\label{cor:bias}
For every $s>0$ and $k\ge2$,
$$
a_k(\sigma_c;s)
=
\frac{s}{k!}\He_{k-2}(c/s)\varphi(c/s).
$$
Moreover, $a_0(\sigma_c;s)$ and $a_1(\sigma_c;s)$
are strictly positive.
\end{corollary}

\begin{proof}
Put $u=c/s$. For $k\ge2$,
$$
a_k(\sigma_c;s)=\frac{s}{k!}\int_u^\infty (z-u)\He_k(z)\varphi(z)\,dz
=\frac{s}{k!}\He_{k-2}(u)\varphi(u),
$$
where the second equality follows from two integrations by parts using $(\He_j\varphi)'=-\He_{j+1}\varphi$. The cases $k=0,1$ follow directly from the same integral and are strictly positive.
\end{proof}

\begin{corollary}
Fix $A,R\neq0$ and a degree bound
$D\in\N_0$. Let $s$ be the common coordinate scale.
For all but finitely many thresholds $c\in\R$,
$$
\Pi_kF_c\ne0
\quad\Longleftrightarrow\quad
Q_k\ne0
\qquad(0\le k\le D).
$$
\end{corollary}

\begin{proof}
The coefficients in degrees zero and one never vanish.
For $2\le k\le D$, vanishing occurs precisely when
$c/s$ is a root of $\He_{k-2}$. Each of these finitely
many polynomials has finitely many roots. Outside the
resulting finite set of thresholds, every activation
coefficient through degree $D$ is nonzero.
Apply \cref{thm:master}.
\end{proof}

Thus a suitable shared bias retains every nonzero
coordinate polynomial through a chosen degree.
It cannot recover a degree for which $Q_k=0$, or select
additional polynomial directions with $A$ and $R$ fixed.

Unlike the positively homogeneous case, changing the
input scale at a fixed nonzero threshold can change
activation support, because it changes $c/s$. This
concerns only the scalar activation coefficients, so any change
to the lift must also be accounted for in the coordinate
polynomials.

\subsection{An activation obstruction for
\texorpdfstring{$C_8$}{C8}}
\label{sec:C8}

The ungraded ReLU obstruction in this example is already
part of the cyclic-group classification in
\cite[Theorem~3B.4]{GTW2025}. Here we identify its first
permitted degree and explain the effect of a bias.

The real regular representation of $C_8$ decomposes as
$$
\R[C_8]
\cong
\one\oplus\varepsilon
\oplus\rho_1\oplus\rho_2\oplus\rho_3,
$$
where $\varepsilon$ is the sign character and $\rho_j$
is the two-dimensional representation in which a generator
acts by rotation through $2\pi j/8$.

Consider the two projected layers from $\rho_1$ to $\rho_3$
and from $\rho_3$ to $\rho_1$.

The element $a^4$ acts as $-I$ on $\rho_1$ and on $\rho_3$, so
\cref{cor:parity-existence} already shows that only odd degrees occur in
either direction. \Cref{ex:cyclic} locates the first of them:
$$
D(\rho_1,\rho_3)
=
D(\rho_3,\rho_1)
=
\{3,5,7,\ldots\}.
$$
Indeed, for the first direction the character condition is
$r\equiv\pm3\pmod8$, while for the reverse direction it is
$3r\equiv\pm1\pmod8$, which is equivalent to the same
condition. Together with $|r|\le k$ and
$r\equiv k\pmod2$, this permits exactly the odd degrees
$k\ge3$.

By \cref{prop:abelian-realisation}, the coordinate
polynomials are nonzero in every permitted degree.
For a positively homogeneous activation, however,
\cref{thm:classification} gives
$$
a_k(\sigma;s)=0
\qquad(k=3,5,7,\ldots).
$$
Hence every chaos component of either projected layer
vanishes. Since these layers are continuous, Gaussian
almost-everywhere vanishing implies that they vanish
identically.

A shifted ReLU with $c\ne0$ changes this conclusion.
Its third coefficient is nonzero by \cref{cor:bias},
so both projected layers acquire a nonzero third-chaos
component.

This distinguishes the example from
\cref{ex:strict-coordinate}. There, the quadratic
coordinate polynomial vanishes and changing the activation
cannot repair it. Here, the cubic coordinate polynomial
is already nonzero and the bias supplies its missing scalar
coefficient.

\vspace{0.7cm}
\section{Architectures determined by the degree calculation}\label{sec:architecture}

The preceding results concern one component of one layer. We now describe what a collection of such components can span, how to control its activation coefficients, and how these statements lead to explicit polynomial layers.

\subsection{The degree subspace of a fixed collection of components}
At a fixed degree, each branch contributes at most one polynomial direction. We first compute the span of those directions.

Fix equivariant $A_h:L\to H_h$ and $R_h:H_h\to K$ for $1\le h\le M$, where each $H_h=\R^{\cB_h}$ has a transitive coordinate basis $\cB_h$ and $A_h\ne0$. Write
$$
Q_{h,k}(x)=R_h\bigl((A_hx)^{\odot k}\bigr),\qquad
\mathcal S_k:=\Span\{Q_{h,k}:1\le h\le M\}\subseteq\Poly_k(L,K)^G.
$$
Consider the layer $F_\theta(x)=\sum_h\theta_hR_h(\sigma_h)_{\cB_h}(A_hx)$, with independent real output coefficients $\theta_h$.

\begin{proposition}\label{prop:accessible}
If $a_k(\sigma_h;s_h)\ne0$ for all $h$, where $s_h^2=\tr(A_hA_h^*)/\dim H_h$, then
$$
\{\Pi_kF_\theta:\theta\in\R^M\}=\Pi_k\mathcal S_k.
$$
Its dimension is the rank of the positive semidefinite Gram matrix
$$
\mathsf G^{(k)}_{h\ell}
=k!\,\tr\left(R_h(A_hA_\ell^*)^{\circ k}R_\ell^*\right).
$$
Here $A_hA_\ell^*:H_\ell\to H_h$ is represented in the two coordinate bases; its entrywise power is therefore well-defined even when $H_h$ and $H_\ell$ have different dimensions. If an activation coefficient vanishes, omit that component before forming the span and Gram matrix. In particular, the whole degree-$k$ equivariant space is spanned exactly when this rank equals $\dim\Hom_G(\Sym^kL,K)$.
\end{proposition}

\begin{proof}
Fix the degree $k$. By \cref{thm:master}, the $h$th branch contributes
$$
\Pi_k\!\left(R_h(\sigma_h)_{\cB_h}(A_hx)\right)
=
\frac{a_k(\sigma_h;s_h)}{s_h^k}\,\Pi_kQ_{h,k}.
$$
Thus, when $a_k(\sigma_h;s_h)\ne0$, the activation only rescales the
degree-$k$ direction $\Pi_kQ_{h,k}$ selected by that branch. Since the
coefficients $\theta_h$ are independent, varying $\theta$ therefore gives
exactly
$$
\{\Pi_kF_\theta:\theta\in\R^M\}
=
\Span\{\Pi_kQ_{h,k}:1\le h\le M\}
=
\Pi_k\mathcal S_k.
$$

It remains to determine the dimension of this span. The matrix
$\mathsf G^{(k)}$ is precisely the Gram matrix of the vectors
$\Pi_kQ_{h,k}$. The same calculation as in \cref{thm:master} gives
$$
\mathsf{G}_{h \ell}^{(k)}
=
k!\,\tr\!\left(
R_h(A_hA_\ell^*)^{\circ k}R_\ell^*
\right), 
\quad \text{and} \quad
\operatorname{rank}\mathsf G^{(k)}
=
\dim \Pi_k\mathcal S_k.
$$
Finally, $\Pi_k$ is an isomorphism on homogeneous degree-$k$
polynomials, so this accessible space equals the full degree-$k$
equivariant polynomial space exactly when
$$
\operatorname{rank}\mathsf G^{(k)}
=
\dim\Hom_G(\Sym^kL,K).
$$
If $a_k(\sigma_h;s_h)=0$, the corresponding branch contributes nothing
in degree $k$ and may simply be omitted.
\end{proof}

This is the degree-$k$ space accessible for the chosen lifts $A_h$ and readouts $R_h$. Varying those maps may enlarge it. Moreover, although each degree can be analysed separately, the same
coefficients $\theta_h$ act across all degrees, so the degree components
cannot in general be chosen independently.

This also yields an approximation obstruction. For
$T\in L^2(\gamma_L;K)^G$ and any degree bound $D\ge0$,
$$
\inf_\theta \norm{T-F_\theta}_{L^2}^2
\ge
\sum_{k=0}^D
\operatorname{dist}_{L^2}\!\left(
\Pi_kT,\Pi_k\mathcal S_k
\right)^2,
$$
since the degree-$k$ component
of $F_\theta$ lies in $\Pi_k\mathcal S_k$. Thus any component of $\Pi_kT$
outside this space gives an unavoidable error for the fixed architecture under the Gaussian reference measure.

The preceding proposition treats one degree at a time, but the same branch coefficient acts across all degrees. To control several degrees independently, we need activations with prescribed initial Hermite coefficients.

\begin{proposition}\label{prop:hermite-control}
Fix $D\ge0$. There exist real $t_0,\dots,t_D$ such that, for every $(b_0,\dots,b_D)\in\R^{D+1}$, there are $\theta_0,\dots,\theta_D\in\R$ for which the continuous piecewise-linear function
$$
\sigma(z)=\sum_{r=0}^D\theta_r \operatorname{ReLU}(z-t_r)
$$
satisfies $a_k(\sigma;1)=b_k$ for $0\le k\le D$. Thus the first $D+1$ Hermite coefficients can be chosen independently.
\end{proposition}

\begin{proof}
For $0\le k\le D$, set $f_k(t):=a_k(\operatorname{ReLU}(\,\cdot-t))$. By linearity of the Hermite coefficients, if $\sigma(z)=\sum_{r=0}^D\theta_r\operatorname{ReLU}(z-t_r)$, then
$$
a_k(\sigma;1)=\sum_{r=0}^D\theta_r f_k(t_r).
$$
It is therefore enough to show that $(f_0(t),\dots,f_D(t))$, as $t$ varies, spans $\R^{D+1}$. If not, there exist coefficients $c_0,\dots,c_D$, not all zero, such that
$\sum_{k=0}^D c_kf_k(t)=0$ for every $t$. Differentiating $f_k(t)$ twice gives $f_k''(t)=\frac1{k!}\He_k(t)\varphi(t)$. Hence
$$
0=\sum_{k=0}^D c_kf_k''(t)
=\varphi(t)\sum_{k=0}^D\frac{c_k}{k!}\He_k(t).
$$
Since $\varphi(t)>0$ and the Hermite polynomials are linearly independent, $c_k=0$ for every $k$.

Therefore the vectors $(f_0(t),\dots,f_D(t))$ span $\R^{D+1}$. Choose $t_0,\dots,t_D$ giving a basis. Then, for any prescribed
$(b_0,\dots,b_D)$, suitable coefficients $\theta_r$ give
$a_k(\sigma;1)=b_k$ for all $0\le k\le D$.
\end{proof}

For a lift with coordinate variance $s^2$, the preceding proposition may
be applied after rescaling the hidden coordinates by $s$. Choosing
$b_k=\delta_{kj}$ produces a piecewise-linear activation whose chaos components vanish in every degree $0,\dots,D$ except degree $j$.
Components above degree $D$ may still be present.

There are two simpler ways to obtain an exact degree-$j$ layer. Using the
polynomial activation $\He_j$ coordinatewise gives the pure $j$th-chaos component
$$
R(\He_j)_{\cB}(Ax/s)
=
s^{-j}\Pi_j\bigl[R\bigl((Ax)^{\odot j}\bigr)\bigr].
$$
Alternatively, $R\bigl((Ax)^{\odot j}\bigr)$
is itself an ordinary homogeneous polynomial of degree $j$.

Thus the piecewise-linear construction controls only finitely many chaos
coefficients, whereas the Hermite activation isolates a single chaos
degree exactly. In particular, \cref{prop:hermite-control} does not control the higher-degree tail.

\subsection{Regular lifts span every degree}
Regular representations of finite groups give a simple construction showing
that coordinatewise powers can realise every equivariant polynomial map of a
fixed degree. The obstruction of \cref{ex:strict-coordinate} is therefore a
property of the chosen coordinates, not an intrinsic one.

\begin{proposition}\label{prop:regular-span}
Let $G$ be finite and let $L,K$ be real orthogonal $G$-representations. For
$u\in L$ define the linear map
$$
A_u:L\longrightarrow\R G,
\qquad
(A_ux)_g=\ip{x}{gu},
$$
where $\R G$ carries the left regular action, and for $v\in K$ define $A_v:K\to\R G$ analogously and put $R_v:=|G|^{-1}A_v^*$, that is, $R_vz=|G|^{-1}\sum_{g\in G}z_g\,gv$. Then $A_u$
and $R_v$ are equivariant, their coordinate polynomial is
$$
Q_{u,v,k}(x)
=
R_v\bigl((A_ux)^{\odot k}\bigr)
=
\frac1{|G|}\sum_{g\in G}\ip{x}{gu}^k\,gv ,
$$
and for every $k\ge0$ these polynomials span $\Poly_k(L,K)^G$ as $u$ and $v$
vary.
\end{proposition}

\begin{proof}
Equivariance of $A_u$ is the identity
$(A_u(hx))_g=\ip{x}{h^{-1}gu}=(A_ux)_{h^{-1}g}$, and $R_v$ is a multiple of the
adjoint of an equivariant map, hence equivariant. Since $A_v^*e_g=gv$, the
displayed formula for $Q_{u,v,k}$ follows.

Pure powers of linear forms span the homogeneous scalar polynomials of degree
$k$, so the maps $P_{u,v}(x)=\ip{x}{u}^kv$ span $\Poly_k(L,K)$. Averaging, that is,
$$
\mathcal R(P)(x)=|G|^{-1}\sum_{g}gP(g^{-1}x),
$$
projects
$\Poly_k(L,K)$ onto $\Poly_k(L,K)^G$, and orthogonality of the action gives
$\mathcal R(P_{u,v})=Q_{u,v,k}$. A projection maps a spanning set of the
ambient space to a spanning set of its image.
\end{proof}

The proposition gives a universal, but generally inefficient, way to match any prescribed equivariant polynomial components through a finite degree
bound $D$. For each required degree $k$, choose finitely many
$Q_{u,v,k}$ spanning the desired component, with $u$ normalised so that the
coordinates of $A_uX$ have variance one.

By \cref{prop:hermite-control}, one may choose a piecewise-linear activation
whose Hermite coefficients vanish in all degrees $0,\dots,D$ except degree
$k$. Then \eqref{eq:master} shows that the corresponding layer contributes,
up to degree $D$, only the chosen degree-$k$ coordinate polynomial. Summing
over the required degrees gives the prescribed chaos components.

Thus copies of the regular representation are sufficient to recover every
equivariant polynomial direction. The cost is that each regular block has $|G|$ hidden coordinates, and spanning several directions may require several such blocks.

\vspace{0.7cm}
\section{Subset representations of the symmetric group}
\label{sec:subsets}

We now make the ambient polynomial spaces from \cref{sec:gaussian} explicit for an important family of permutation representations. Subset representations model vertex, edge, and higher-order subset features in equivariant graph networks. This section describes the full polynomial spaces against which the coordinate-selected spaces of \cref{sec:coordinatewise,sec:architecture} can be compared.

\begin{definition}
Let $[n]=\{1,\dots,n\}$, with $[0]=\varnothing$,
and let $S_n$ be its permutation group. For $r\ge0$, define
$$
X_{n,r}:=\binom{[n]}r
=
\{A\subseteq[n]:|A|=r\}.
$$
The \emph{$r$-subset representation} is the real vector space
$$
V_{n,r}
:=
\R[X_{n,r}]
=
\bigoplus_{A\in X_{n,r}}\R e_A,
$$
equipped with the inner product for which the basis
$\{e_A:A\in X_{n,r}\}$ is orthonormal. The action of $S_n$ is defined by
$g\cdot e_A=e_{gA}$, where $gA=\{g(a):a\in A\}$.
\end{definition}

The space $V_{n,0}\cong\R$ is the trivial representation. Intuitively, $V_{n,1}$ consists of vertex features, $V_{n,2}$ of features on unordered pairs of distinct
vertices, and $V_{n,3}$ of features on unordered triples.
Pair features may be viewed as real edge weights. Triple features are assigned to all triples, whether or not those triples form triangles in a particular graph.
Throughout, these spaces contain arbitrary real-valued features on every subset of the indicated size.

Fix an input subset size $r\ge1$, an output subset size
$s\ge0$, and a polynomial degree $d\ge0$. Our objective
is to construct an explicit basis of
$$
\Poly_d(V_{n,r},V_{n,s})^{S_n}
\cong
\Hom_{S_n}(\Sym^dV_{n,r},V_{n,s}).
$$

An element of this space has the form
$$
P(x)=\sum_{Q\in X_{n,s}}P_Q(x)e_Q,
$$
where each $P_Q$ is a homogeneous polynomial of degree
$d$ in the input coordinates $x_A$. Equivariance means $P(g\cdot x)=g\cdot P(x)$, or equivalently, $P_{gQ}(g\cdot x)=P_Q(x)$ for $g\in S_n,\ Q\in X_{n,s}$.

A degree-$d$ input monomial
$x_{A_1}\cdots x_{A_d}$ records $d$ input subsets,
with repetitions allowed. An output coordinate also
specifies an $s$-subset $Q$. The basis constructed below
is indexed by these configurations up to simultaneous
relabelling of the vertices. Each basis map sums the
monomials belonging to one such configuration type.

\subsection{Symmetric powers are multihypergraph modules}

To describe symmetric powers combinatorially, for a finite $G$-set $X$ and $d\ge0$, let $\MSet_d(X)$ denote the size-$d$ multisets on $X$ with the induced action.

\begin{proposition}
\label{prop:sym-mset}
For a finite group $G$ and a finite $G$-set $X$,
$$
\Sym^d(\R[X])\cong\R[\MSet_d(X)]
$$
as $G$-representations. If $m_\omega$ represents an orbit
$\omega\in\MSet_d(X)/G$ and
$H_\omega=\operatorname{Stab}_G(m_\omega)$, then
$$
\Sym^d(\R[X])
\cong
\bigoplus_{\omega\in\MSet_d(X)/G}
\Ind_{H_\omega}^G\one.
$$
\end{proposition}

\begin{proof}
The monomials $e^m=\prod_{x\in X}e_x^{m(x)}$ for $ m\in\MSet_d(X)$
form a basis of $\Sym^d(\R[X])$, and $g e^m=e^{gm}$.
Each orbit therefore spans a permutation module
isomorphic to $\R[G/H_\omega]=\Ind_{H_\omega}^G\one$.
\end{proof}

\begin{definition}
An \emph{$r$-uniform multihypergraph on $[n]$ with $d$ edges} is a
multiset
$$
\mathcal H=\{E_1,\dots,E_d\},
\qquad
E_i\in X_{n,r},
$$
where repeated edges are allowed.
\begin{itemize}
    \item A vertex is \emph{active} if it lies in at least one edge of
    $\mathcal H$. The subhypergraph obtained
    by deleting all isolated vertices is called the \emph{active core}.
    
    \item Two multihypergraphs are \emph{isomorphic} if there is a bijection
    between their vertex sets that preserves edge multiplicities. We write
    $\Aut(\mathcal H)$ for the automorphism group of $\mathcal H$. 
    
    \item Let $\mathfrak H_{r,d}$ denote the set of isomorphism classes of
    $r$-uniform multihypergraphs with $d$ edges and no isolated vertices. For $\mathcal H\in\mathfrak H_{r,d}$, write $v(\mathcal H)$ for its number of vertices.
\end{itemize}
\end{definition}

Take $X=X_{n,r}$. A monomial
$e_{A_1}\cdots e_{A_d}$ corresponds to the edge
multiset $\{A_1,\dots,A_d\}$, and two monomials lie in the same
$S_n$-orbit exactly when their active cores are isomorphic.
Every core satisfies $v(\mathcal H)\le rd$,
with equality precisely when its edges are pairwise disjoint.

For $v(\mathcal H)\le n$, let $P_{\mathcal H}(n)$ denote
the span of all monomials whose active core is isomorphic
to $\mathcal H$.

The active core is exactly the orbit data, so it gives the desired decomposition of the symmetric power.

\begin{theorem}
\label{thm:hypergraph-decomposition}
For $n\ge0$, $r\ge1$ and $d\ge0$,
$$
\Sym^dV_{n,r}
=
\bigoplus_{\substack{[\mathcal H]\in\mathfrak H_{r,d}\\
                    v(\mathcal H)\le n}}
P_{\mathcal H}(n), \qquad \text{where} \qquad P_{\mathcal H}(n)
\cong
\Ind_{\Aut(\mathcal H)\times S_{n-v(\mathcal H)}}^{S_n}
\one.
$$
The subgroup acts by automorphisms on a labelled copy of
the core and by arbitrary permutations on its complement.
For $n\ge rd$, every core type occurs.

For each partition $\lambda\vdash n$, the multiplicity
of the irreducible Specht module $S^\lambda$ is
$$
[P_{\mathcal H}(n):S^\lambda]
=
\dim (S^\lambda)^{\Aut(\mathcal H)\times S_{n-v(\mathcal H)}}.
$$
\end{theorem}

\begin{proof}
See \cref{app:hypergraph-decomposition}.
\end{proof}

This is a decomposition into orbit spans, not generally
into irreducible representations. Each summand records
one overlap pattern among the input subsets. Importantly, the indexing set stabilises for $n\ge rd$, although the modules
$P_{\mathcal H}(n)$ themselves still depend on $n$.

\begin{example} \label{ex:r2d2-types}
For $r=2$, $d=2$ and $n\ge4$, there are three core types:
a repeated edge, two adjacent edges and two disjoint edges.
Writing $e_{ab}=e_{\{a,b\}}$, the representatives are $e_{12}^2$, $e_{12}e_{23}$ and $e_{12}e_{34}$, respectively. Hence
$$
\Sym^2V_{n,2} = P_{\mathrm{rep}}(n) \oplus P_{\mathrm{adj}}(n) \oplus P_{\mathrm{disj}}(n).
$$
For general $r$, the orbit of $e_Ae_B$ is determined by
$t=|A\cap B|$, where $\max(0,2r-n)\le t\le r$. If $t<r$, its stabiliser is $S_t\times(S_{r-t}\wr S_2)\times S_{n-2r+t}$.
The factors permute the intersection, the two exclusive
parts (allowing their exchange), and the unused vertices.
If $t=r$, the edge is repeated and the stabiliser is
$S_r\times S_{n-r}$.
\end{example}

Symmetric powers of permutation modules have also been studied in the modular setting \cite{Jiang2021}. Here the orbit description is useful because it keeps the overlap data visible without first decomposing into irreducibles. For example, at $n=rd$ the
disjoint-edge summand is the Foulkes module $\Ind_{S_r\wr S_d}^{S_{rd}}\one$, whose general irreducible decomposition is an open
problem \cite{PagetWildon2016}.

\subsection{Polynomial equivariants: the marked-orbit basis}

The marked-orbit basis describes all homogeneous polynomial maps from $r$-subset to $s$-subset features. We begin with the classical linear case.

\begin{proposition}\label{prop:orbit-kernel}
For finite $G$-sets $X,Y$ we have $\Hom_G(\R[X],\R[Y])\cong\R[X\times Y]^G$, with basis the maps
$$
T_\cO(e_x)=\sum_{y:(x,y)\in\cO}e_y
$$
indexed by the $G$-orbits $\cO\subseteq X\times Y$.
\end{proposition}
\begin{proof}
A linear map $T$ with matrix $k_T(x,y)$ is equivariant if and only if $k_T(gx,gy)=k_T(x,y)$, that is, if and only if $k_T$ is constant on diagonal orbits. See also \cite[Lemma~2E.3]{GTW2025} and \cite{Maron2019IGN}.
\end{proof}

By \cref{prop:sym-mset} we may apply this with $X=\MSet_d(\Xnr)$, which leads to the following description of the orbits.

\begin{definition}\label{def:marked}
A \emph{marked degree-$d$ $(r,s)$-interaction} is a pair $(\mathcal A,Q)$ with $\mathcal A\in\MSet_d(\Xnr)$ and $Q\in\Xns$, that is, an $r$-uniform $d$-edge multihypergraph on $[n]$ together with a distinguished $s$-subset of vertices. Its \emph{type} is its orbit under simultaneous relabelling. Marked vertices need not be active.
\end{definition}

\begin{theorem}\label{thm:marked-basis}
$\Poly_d(\Vnr,\Vns)^{S_n} \cong \Hom_{S_n}(\Sym^d\Vnr,\Vns)$, with a basis indexed by the diagonal $S_n$-orbits $\cO\subseteq \MSet_d(\Xnr)\times\Xns$. For such an orbit, define
$$
\Psi_\cO(x)
:=
\sum_{Q\in\Xns}
\left(
\sum_{\mathcal A:(\mathcal A,Q)\in\cO}
x^{\mathcal A}
\right)e_Q,
\qquad
x^{\mathcal A}
:=
\prod_{A\in\Xnr}x_A^{m_{\mathcal A}(A)}.
$$
Then the maps $\Psi_\cO$ form a basis of $\Poly_d(\Vnr,\Vns)^{S_n}$.
\end{theorem}

\begin{proof}
By \cref{prop:sym-mset}, a degree-$d$ monomial is indexed by a
multiset $\mathcal A\in\MSet_d(\Xnr)$. Hence a polynomial map
$\Vnr\to\Vns$ is determined by coefficients
$c_{\mathcal A,Q}$ indexed by $(\mathcal A,Q)\in\MSet_d(\Xnr)\times\Xns$.
By \cref{prop:orbit-kernel}, such a map is $S_n$-equivariant exactly
when these coefficients are constant on diagonal $S_n$-orbits.
The corresponding orbit indicators are precisely the maps
$\Psi_\cO$.
\end{proof}

Under the polarisation identification of \cref{prop:poly-sym}, the usual multinomial factors only rescale the multiset basis vectors. These factors are constant on $S_n$-orbits, so they do not change the orbit indexing or the basis statement.

Orbit-indexed graph-polynomial bases also appear in \cite{Puny2023}.The orbit basis is explicit, but we also want its dimension and stable range. Fixing the active core reduces the count to the possible positions of the output mark.

\begin{proposition}
\label{prop:marked-count}
Fix an active core $\cH$ with vertex set $U=V(\cH)$, and write $v=v(\cH)$.
A marked type with core $\cH$ is determined by $q:=|Q\cap U|$
together with the $\Aut(\cH)$-orbit of the active part $Q\cap U\in\binom{U}{q}$.
Consequently,
\begin{equation}\label{eq:finite-dim}
\dim\Hom_{S_n}(\Sym^d\Vnr,\Vns)
=
\sum_{\substack{[\cH]\in\mathfrak H_{r,d}\\ v(\cH)\le n}}
\;
\sum_{q=\max(0,\,s-(n-v(\cH)))}^{\min(s,\,v(\cH))}
\left|
\binom{V(\cH)}{q}\Big/\Aut(\cH)
\right|.
\end{equation}
\end{proposition}

\begin{proof}
By \cref{thm:hypergraph-decomposition} and Frobenius reciprocity,
the contribution of a fixed core $\cH$ is $\dim(\Vns)^{\Aut(\cH)\times S_{n-v}}$.
Since $\Vns=\R[X_{n,s}]$ is a permutation representation, this
dimension is the number of $(\Aut(\cH)\times S_{n-v})$-orbits on
$s$-subsets $Q\subseteq[n]$.

Write $q=|Q\cap U|$. $\Aut(\cH)$ acts on the active part $Q\cap U$, while $S_{n-v}$ is transitive on all $(s-q)$-subsets of the inactive vertices. Thus, for fixed $q$,
the orbit of $Q$ is determined exactly by the $A$-orbit of
$Q\cap U\in\binom{U}{q}$.

Such a marking exists precisely when $0\le q\le v$ and $0\le s-q\le n-v$, giving the stated limits. Summing over the core types proves \eqref{eq:finite-dim}.
\end{proof}

By Burnside's lemma, evaluating \eqref{eq:finite-dim} requires only the cycle structures of the automorphism groups. \Cref{thm:molien-polya} computes the same number by summing over conjugacy classes of $S_n$ instead.

The counting formula also shows exactly when increasing $n$ can no longer create new marked types.

\begin{theorem}\label{prop:sharp}
Let $r,d\ge1$ and $s\ge0$, and write $D(n):=\dim\Hom_{S_n}(\Sym^d\Vnr,\Vns)$. Then:

\begin{enumerate}
    \item $D(n)$ is non-decreasing in $n$.

    \item For $n\ge rd+s$,
    $$
    D(n)
    =
    \sum_{[\cH]\in\mathfrak H_{r,d}}
    \sum_{q=0}^{\min(s,v(\cH))}
    \left|
    \binom{V(\cH)}q\Big/\Aut(\cH)
    \right|,
    $$
    so $D(n)$ is independent of $n$ in this range.

    \item The threshold is sharp: $D(rd+s-1)<D(rd+s)$.
\end{enumerate}

In the stable range, the basis $\{\Psi_\cO\}$ is indexed by a set
independent of $n$, so one coefficient vector defines one equivariant
polynomial layer simultaneously for every $n\ge rd+s$.
\end{theorem}

\begin{proof}
In \eqref{eq:finite-dim} the set of admissible types grows with $n$ and each lower summation limit is non-increasing in $n$, so $D$ is non-decreasing. For $n\ge rd+s$ we have $n-v\ge s$ for every core, so every lower limit is $0$ and no term depends on $n$, which gives \emph{(b)}.

For sharpness, let $\cH_0$ be the disjoint type, with $v(\cH_0)=rd$. At $n=rd+s-1$ the value $q=0$ would require $s-q\le n-v=s-1$, which fails, so $q=0$ is inadmissible for $\cH_0$. Also when $s=0$ the core $\cH_0$ does not embed at all. At $n=rd+s$ it becomes admissible and contributes exactly one further orbit, while no other term decreases. Hence the strict inequality.
\end{proof}

\subsection{Worked example: seven quadratic edge-to-vertex maps}\label{sec:seven}

Take $(r,d,s)=(2,2,1)$ and write $x_{uv}=x_{\{u,v\}}$. By \cref{thm:marked-basis} we must list the marked types: a repeated edge with the mark on it or off it; an adjacent pair with the mark at the centre, at an endpoint, or outside; a disjoint pair with the mark on an active vertex or outside. These give seven types (\cref{fig:seven-types}), hence $\dim\Hom_{S_n}(\Sym^2V_{n,2},V_{n,1})=7$ for $n\ge5$. At $n=3$ only four types fit, while at $n=4$ only the disjoint pair with an outside mark is absent. Thus the dimensions are $4,6,7,7,\dots$ for $n=3,4,5,6,\dots$.

\begin{figure}[H]
\centering
\begin{tikzpicture}[scale=1, every node/.style={font=\small}, vertex/.style={circle,draw,inner sep=1.4pt}, markedv/.style={circle,draw,fill=black,inner sep=2.1pt}]
\begin{scope}[shift={(0,0)}]
\node[markedv] (a) at (0,0) {}; \node[vertex] (b) at (1,0) {}; \draw[double] (a)--(b);
\node at (.5,-.55) {\shortstack{$\Phi_1$\\repeated: on}};
\end{scope}
\begin{scope}[shift={(3.2,0)}]
\node[vertex] (a) at (0,0) {}; \node[vertex] (b) at (1,0) {}; \node[markedv] at (.5,.65) {}; \draw[double] (a)--(b);
\node at (.5,-.55) {\shortstack{$\Phi_2$\\repeated: off}};
\end{scope}
\begin{scope}[shift={(6.4,0)}]
\node[vertex] (a) at (0,0) {}; \node[markedv] (b) at (.7,.35) {}; \node[vertex] (c) at (1.4,0) {}; \draw (a)--(b)--(c);
\node at (.7,-.55) {\shortstack{$\Phi_3$\\adjacent: centre}};
\end{scope}
\begin{scope}[shift={(10,0)}]
\node[markedv] (a) at (0,0) {}; \node[vertex] (b) at (.7,.35) {}; \node[vertex] (c) at (1.4,0) {}; \draw (a)--(b)--(c);
\node at (.7,-.55) {\shortstack{$\Phi_4$\\adjacent: endpoint}};
\end{scope}
\begin{scope}[shift={(1.2,-2.0)}]
\node[vertex] (a) at (0,0) {}; \node[vertex] (b) at (.7,.35) {}; \node[vertex] (c) at (1.4,0) {}; \node[markedv] at (.7,0.75) {}; \draw (a)--(b)--(c);
\node at (.7,-.55) {\shortstack{$\Phi_5$\\adjacent: outside}};
\end{scope}
\begin{scope}[shift={(5.2,-2.0)}]
\node[markedv] (a) at (0,0) {}; \node[vertex] (b) at (1,0) {}; \node[vertex] (c) at (2,0) {}; \node[vertex] (d) at (3,0) {}; \draw (a)--(b); \draw (c)--(d);
\node at (1.5,-.55) {\shortstack{$\Phi_6$\\disjoint: active}};
\end{scope}
\begin{scope}[shift={(9.4,-2.0)}]
\node[vertex] (a) at (0,0) {}; \node[vertex] (b) at (1,0) {}; \node[vertex] (c) at (2,0) {}; \node[vertex] (d) at (3,0) {}; \node[markedv] at (1.5,.75) {}; \draw (a)--(b); \draw (c)--(d);
\node at (1.5,-.55) {\shortstack{$\Phi_7$\\disjoint: outside}};
\end{scope}
\end{tikzpicture}
\caption{The seven marked interaction types for quadratic maps $V_{n,2}\to V_{n,1}$ in the stable range. Filled vertices carry the output mark $v$; double lines denote a repeated input edge.}
\label{fig:seven-types}
\end{figure}
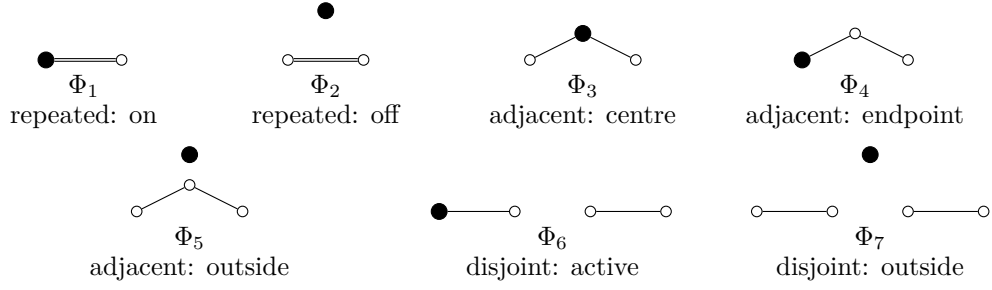
Writing the $v$th output coordinate gives the following basis, using the same labels as \cref{fig:seven-types}.

\begin{align*}
(\Phi_1x)_v&=\sum_{u\ne v}x_{uv}^2,
&\text{repeated edge, mark on},\\
(\Phi_2x)_v&=\sum_{\{i,j\}\not\ni v}x_{ij}^2,
&\text{repeated edge, mark off},\\
(\Phi_3x)_v&=\sum_{\{a,c\}\subseteq[n]\setminus\{v\}}x_{av}x_{vc},
&\text{adjacent edges, centre},\\
(\Phi_4x)_v&=\sum_{u\ne v}\ \sum_{w\notin\{u,v\}}x_{vu}x_{uw},
&\text{adjacent edges, endpoint},\\
(\Phi_5x)_v&=\sum_{b\ne v}\ \sum_{\{a,c\}\subseteq[n]\setminus\{v,b\}}x_{ab}x_{bc},
&\text{adjacent edges, outside},\\
(\Phi_6x)_v&=\sum_{u\ne v}\ \sum_{\{a,b\}\subseteq[n]\setminus\{v,u\}}x_{vu}x_{ab},
&\text{disjoint edges, active},\\
(\Phi_7x)_v&=\sum_{\substack{\{E_1,E_2\}\\E_1\cap E_2=\varnothing\\v\notin E_1\cup E_2}}x_{E_1}x_{E_2},
&\text{disjoint edges, outside}.
\end{align*}
In the last sum, $E_1$ and $E_2$ are two-element subsets of $[n]$ and their pair is unordered. Each map is an orbit sum of one marked type. Their monomial supports are disjoint, so the maps are independent and the dimension count shows that they form a basis.

The basis also shows the geometry directly. $\Phi_1$ is a local square statistic at $v$, while $\Phi_3$ and $\Phi_4$ are quadratic contractions along length-two paths through or ending at $v$. The other four use configurations that are not local path contractions at $v$.

\FloatBarrier

\subsection{Evaluation of the seven quadratic operations}\label{sec:quadratic-evaluation}
For implementation, the seven orbit sums can be evaluated without enumerating triples, quadruples or five-tuples. Write $X_{uv}=x_{uv}$ for $u\ne v$ and $X_{vv}=0$, and define
$$
h=X\one,\qquad q_v=\sum_{u\ne v}x_{uv}^2,
\qquad S=\sum_{u<v}x_{uv},\qquad T=\sum_{u<v}x_{uv}^2.
$$
Let
$$
W=\frac12\sum_v(h_v^2-q_v),\qquad
D_\mathrm{disj}=\frac12(S^2-T)-W.
$$
These are the total adjacent-pair and disjoint-pair products, respectively. All quantities refer to scalar, real-valued edge features.

\begin{proposition}\label{prop:quadratic-fast}
The seven basis maps of \cref{sec:seven} satisfy
\begin{align*}
(\Phi_1x)_v&=q_v,&
(\Phi_2x)_v&=T-q_v,\\
(\Phi_3x)_v&=\tfrac12(h_v^2-q_v),&
(\Phi_4x)_v&=(Xh)_v-q_v,\\
(\Phi_5x)_v&=W-(\Phi_3x)_v-(\Phi_4x)_v,\\
(\Phi_6x)_v&=S h_v-h_v^2-(Xh)_v+q_v,&
(\Phi_7x)_v&=D_\mathrm{disj}-(\Phi_6x)_v.
\end{align*}
For $n\ge5$, every homogeneous quadratic equivariant map $V_{n,2}\to V_{n,1}$ is therefore a unique linear combination of these seven maps and can be evaluated in $O(n^2)$ arithmetic operations. With $m$ nonzero edge features supplied sparsely, the cost is $O(n+m)$.
\end{proposition}
\begin{proof}
Expanding $h_v^2$ separates repeated edges from unordered pairs of distinct edges incident to $v$, giving $\Phi_3$. Expanding $(Xh)_v$ separates the returning path $v,u,v$, which contributes $q_v$, from paths with three distinct vertices, giving $\Phi_4$. Every adjacent edge pair places $v$ at its centre, at an endpoint, or outside. Subtracting the first two contributions from $W$ gives $\Phi_5$.

For a fixed edge $\{v,u\}$, the sum of features on edges disjoint from it is $S-h_v-h_u+x_{vu}$. Multiplying by $x_{vu}$ and summing over $u$ gives $\Phi_6$. The total product over unordered distinct edge pairs is $(S^2-T)/2$. Subtract adjacent pairs to obtain $D_\mathrm{disj}$, then subtract those containing $v$ to obtain $\Phi_7$. The first two identities are immediate from the definitions. Only edge sums, the two matrix--vector products $X\one$ and $Xh$, and vertexwise arithmetic are required. Completeness follows from \cref{thm:marked-basis}.
\end{proof}

For dense inputs, the $O(n^2)$ cost is of the same order as reading the $\binom n2$ edge coordinates; for sparse inputs, the $O(n+m)$ bound is linear in the stored input size up to the vertex term.

An explicit quadratic layer is
$$
F_\theta(x)=\sum_{a=1}^7\theta_a\Phi_a(x).
$$
Together with the equivariant linear terms and constant output, this
spans all equivariant polynomial maps of degree at most two. The seven-dimensional statement concerns the quadratic polynomial space on the
full real feature space. The quadratic products in $\Phi_a$ are explicit
operations of the architecture, so this is not a claim that finite ReLU
networks can exactly realise quadratic functions on
$\R^{\binom n2}$. If inputs are restricted to Boolean adjacency matrices,
additional identities such as $x_{uv}^2=x_{uv}$ may introduce dependencies
between polynomial expressions.

\vspace{0.7cm}
\section{Scope and future directions}\label{sec:limitations}

\subsection*{Limitations}
It remains open to characterise which sequences of chaos components come from equivariant piecewise-linear maps; see \cref{cor:pl-chaos}. The transform determines each such map uniquely, but equivariance and square-summability
of the components do not ensure that the resulting map is piecewise linear. A complete structural picture is therefore still missing. Matching finitely many components with shifted ReLUs also leaves the higher-degree tail uncontrolled, so it gives no approximation bound on its own. For compact connected groups, polynomial existence does not suffice to
detect piecewise-linear existence; the additional restrictions are discussed in the introduction.

At the architectural level, the Gram criterion describes the space available in each degree for fixed lifts and readouts. A characterisation of networks with prescribed width and depth, allowing these maps to be trained, remains open here. It must account for shared parameters across degrees and for interactions created by composition and joint processing of source components.

The general spanning construction can also be costly: each regular hidden block has $|G|$ coordinates. For subset features, the multihypergraph orbit modules can still be reducible, so a full irreducible decomposition requires further work. Restricting the polynomial bases to Boolean inputs can introduce additional relations. Finally, the expressivity results do not establish guarantees about training or generalisation.

\subsection*{Constructing architectures}
The main practical task is to turn these constructions into trainable networks with prescribed degree spaces. Choose source and target modules
$L,K$, a degree budget $D$, and equivariant lifts $A_h:L\to H_h$ and readouts $R_h:H_h\to K$,
where each $H_h$ is a transitive permutation module. Use the Gram test in
\cref{prop:accessible} to select coordinate polynomials
$R_h\bigl((A_hx)^{\odot k}\bigr)$ spanning the intended spaces, and choose activations
with nonzero coefficients at the required degrees and scales. Independent
branch weights then parametrise the computed span in each degree. For
separate degree control, \cref{sec:architecture} provides coordinatewise
powers for homogeneous polynomial layers, normalised Hermite activations
for pure chaos layers, and shifted-ReLU combinations for matching
coefficients through $D$. The polynomial constructions need not be
piecewise linear, while finite ReLU coefficient matching leaves an
uncontrolled higher-degree tail.

\subsection*{Graph networks}
The companion paper \cite{CompanionPaper} applies this theory to predicting
graph diameter and algebraic connectivity. The reported gains over a parameter-matched control with
full linear $S_n$-equivariant mixing concern one synthetic benchmark with a frozen GIN backbone. The subset representations studied here place this construction within a broader family of vertex, edge and higher-order features. In each degree, the marked-orbit basis gives all polynomial maps between these spaces, guiding layer construction and tests of which
directions a proposed layer spans (\cref{thm:marked-basis}). Its stable indexing allows coefficients to be shared across graph orders (\cref{prop:sharp}).

A further direction is to adapt the GNN architecture to classify $d$-manifolds from their triangulations. Equivariance under vertex relabelling alone is insufficient: subdivisions and retriangulations,
such as Pachner moves, can change both the vertex set and the incidence structure. Extending the theory to this setting requires specifying how features are transferred between triangulations and how network layers respect these transfers, with the aim of making the final prediction invariant under the chosen changes of triangulation.

\appendix

\section{Supplementary proofs}\label{app:proofs}

This appendix supplies the polarisation and character-theoretic arguments used in the main text.

\subsection{Proof of \texorpdfstring{\cref{prop:poly-sym}}{Proposition (polynomials and symmetric powers)}}\label{app:poly-sym}
\begin{proof}
Polarisation identifies a homogeneous polynomial map $P$ of degree $k$ with a symmetric $k$-linear map $\widetilde P$ satisfying $P(x)=\widetilde P(x,\dots,x)$, and the universal property of $\Sym^kL$ converts $\widetilde P$ into a unique linear $T_P:\Sym^kL\to K$ with $T_P(x^k)=P(x)$. The two constructions are mutually inverse and natural in $L$ and $K$, so $P(gx)=gP(x)$ for all $g$ and $x$ if and only if $T_P$ commutes with the $G$-action.
\end{proof}

\subsection{Proof of \texorpdfstring{\cref{prop:molien}}{Proposition (Molien series)}}\label{app:molien}
\begin{proof}
Diagonalising $\rho_L(g)$ over $\C$ gives $\sum_k\chi_{\Sym^kL}(g)t^k=\det(I-t\rho_L(g))^{-1}$, and taking logarithms,
$$
\log\det\bigl(I-t\rho_L(g)\bigr)^{-1}=\sum_{m\ge1}\frac{t^m}m\tr\rho_L(g^m)=\sum_{m\ge1}\frac{t^m}m\chi_L(g^m).
$$
Real characters are self-conjugate, so $c_k=\langle\chi_{\Sym^kL},\chi_K\rangle$, and averaging over $G$ gives both displayed forms.
\end{proof}

\subsection{Proof of \texorpdfstring{\cref{thm:hypergraph-decomposition}}{Theorem (hypergraph decomposition)}}\label{app:hypergraph-decomposition}
\begin{proof}
A basis of $\Sym^dV_{n,r}$ is given by the monomials
$$
e^m=\prod_{A\in X_{n,r}}e_A^{m(A)},
\qquad
m:X_{n,r}\to\N_0,
\qquad
\sum_A m(A)=d.
$$
The action of $S_n$ is
$$
g\cdot e^m=e^{gm},
\qquad
(gm)(A)=m(g^{-1}A).
$$
Thus each monomial may be identified with an $r$-uniform
multihypergraph on $[n]$ having $m(A)$ copies of the edge $A$.

A permutation $g$ maps the active core of $m$ isomorphically onto that of $gm$. Conversely, any isomorphism of active
cores extends arbitrarily to a permutation of their complements in
$[n]$. Hence two monomials lie in the same $S_n$-orbit if and only if their active cores are isomorphic, and
$$
\Sym^dV_{n,r}
=
\bigoplus_{\substack{[\mathcal H]\in\mathfrak H_{r,d}\\
v(\mathcal H)\le n}}
P_{\mathcal H}(n).
$$

Fix a monomial with active core $\mathcal H$, active vertex set $U$,
and $v=|U|$. Its stabiliser consists precisely of automorphisms of the
core together with arbitrary permutations of the inactive vertices.
Therefore $\operatorname{Stab}_{S_n}(m) \cong \Aut(\mathcal H)\times S_{n-v}$, and the corresponding orbit module is the associated permutation
representation, $P_{\mathcal H}(n) \cong \Ind_{\Aut(\mathcal H)\times S_{n-v}}^{S_n}\one$.

Since $\mathcal H$ has $d$ edges, each of size $r$, $v(\mathcal H)\le rd$. Thus every core type occurs once $n\ge rd$, so the indexing set is
independent of $n$ in this range.

Finally, for $\lambda\vdash n$, absolute irreducibility of the real
Specht modules and Frobenius reciprocity give
$$
\begin{aligned}
[P_{\mathcal H}(n):S^\lambda]
&=
\dim\Hom_{S_n}(P_{\mathcal H}(n),S^\lambda)\\
&=
\dim\Hom_{\Aut(\mathcal H)\times S_{n-v}}
\left(\one,\operatorname{Res}S^\lambda\right)\\
&=
\dim(S^\lambda)^{\Aut(\mathcal H)\times S_{n-v}}.
\end{aligned}
$$
This proves the result.
\end{proof}

\section{Hidden spaces with several coordinate orbits}\label{app:intransitive}

The common scalar in \cref{thm:master} requires equal coordinate variances. With several coordinate orbits there is still an exact formula, but the scalar becomes a diagonal matrix.

Let $A:L\to\R^{\cB}$ and $R:\R^{\cB}\to K$ be equivariant linear maps for an arbitrary permutation basis $\cB$. Put $C=AA^*$ and $s_b=\sqrt{C_{bb}}$. Assume $\sigma(s_b\,\cdot)\in L^2(\gamma_\R)$ whenever $s_b>0$. Let $D_k$ be diagonal with
$$
(D_k)_{bb}=\begin{cases}
a_k(\sigma;s_b)/s_b^k,&s_b>0,\\
\sigma(0),&s_b=0,\ k=0,\\
0,&s_b=0,\ k>0.
\end{cases}
$$
Each $s_b$, and hence each diagonal entry, is constant on a coordinate orbit. Therefore $D_k$ is equivariant.

\begin{proposition}\label{prop:intransitive}
For $F(x)=R\sigma_{\cB}(Ax)$,
$$
\Pi_kF=\Pi_k\bigl[RD_k\bigl((Ax)^{\odot k}\bigr)\bigr],
\qquad
\norm{\Pi_kF}_{L^2}^2=k!\,\tr(RD_kC^{\circ k}D_kR^*).
$$
\end{proposition}
\begin{proof}
Expand each coordinate using its own variance. If $s_b=0$, then $A^*b=0$ and that coordinate is the constant $\sigma(0)$, contributing only to degree zero. For positive variances, the identity from \cref{sec:chaos} gives $\Pi_k\ip{x}{A^*b}^k=s_b^k\He_k(\ip{x}{A^*b}/s_b)$. By \cref{lem:mehler}, the corresponding cross inner product is $k!C_{bb'}^k$, including degree zero with $C^{\circ0}=J$. Summing gives both assertions.
\end{proof}

Different orbit contributions can cancel after the readout. Thus separate nonzero tests for the orbits do not replace this combined formula. In particular, several feature channels in a hidden space do not require a new theory, but do require retaining their possibly different variances and cross terms.

\section{Representation formulae for subset modules}\label{app:subset-representations}

The Specht modules $S^\lambda$, $\lambda\vdash n$, are the irreducible real representations of $S_n$. They are absolutely irreducible, so character inner products below are ordinary multiplicities. For the two-row formula assume $0\le r\le n/2$ and recall that complementation gives $V_{n,r}\cong V_{n,n-r}$ for the remaining nonzero cases.

\subsection{The subset module and its Johnson scheme}
We finish with formulae specific to subset modules. They support the multiplicity-free and computational statements used in the main text.

\begin{proposition}\label{prop:two-row}
$\Vnr\cong\Ind_{S_r\times S_{n-r}}^{S_n}\one$, and for $r\le n/2$
$$
\Vnr\cong\bigoplus_{j=0}^rS^{(n-j,j)} .
$$
In particular $\Vnr$ is multiplicity-free, and $\chi^{(n-j,j)}(g)=\abs{X_{n,j}^{\,g}}-\abs{X_{n,j-1}^{\,g}}$ for $0\le j\le n/2$, with $X_{n,-1}=\varnothing$.
\end{proposition}

\begin{proof}
The action of $S_n$ on $r$-subsets is transitive, with stabiliser $S_r\times S_{n-r}$, so $\Vnr\cong \Ind_{S_r\times S_{n-r}}^{S_n}\one$. The decomposition of this Young permutation module is multiplicity-free and standard. See \cite{James1978,Sagan2001}.

Finally, since $\chi_{V_{n,j}}(g)=|X_{n,j}^{\,g}|$
and $V_{n,j}\cong\bigoplus_{i=0}^j S^{(n-i,i)}$, subtracting the decompositions for $j$ and $j-1$ gives
$$
\chi^{(n-j,j)}(g)
=
|X_{n,j}^{\,g}|-|X_{n,j-1}^{\,g}|.
$$
\end{proof}

Thus $\Vnr$ is transitive and multiplicity-free, and \cref{sec:krein} applies without modification. Here the orbit algebra is the Bose--Mesner algebra of the Johnson scheme $J(n,r)$. Its orbit indicator matrices are indexed by the intersection number $\abs{A\cap A'}$, and the primitive idempotents $E_i$ have entries depending only on that number, given by the dual eigenvalues ($Q$-numbers) of the scheme in the standard normalisation \cite[Section~9.1]{BCN1989}. The Hadamard coefficients of \cref{thm:krein} are therefore obtained from the Johnson Krein parameters in that normalisation.

\subsection{A character formula}\label{sec:subset-character}

\Cref{thm:hypergraph-decomposition} describes the orbit geometry. For dimension and multiplicity calculations, P\'olya theory gives a complementary formula that avoids enumerating hypergraphs up to isomorphism.

\begin{theorem}[Molien--P\'olya formula]\label{thm:molien-polya}
For $g\in S_n$ let
$$
f_r(g):=\abs{\Xnr^{\,g}}=[z^r]\prod_{c\in\Cyc(g)}(1+z^{\abs c}),
$$
the product being over cycles of $g$ on $[n]$, and define $h_d(g)$ by
$$
\sum_{d\ge0}h_d(g)t^d=\exp\Bigl(\sum_{m\ge1}\frac{t^m}{m}f_r(g^m)\Bigr),
\qquad\text{equivalently}\qquad
d\,h_d(g)=\sum_{m=1}^df_r(g^m)h_{d-m}(g).
$$
Then for every $\lambda\vdash n$ and $0\le s\le n$,
\begin{align*}
[\Sym^d\Vnr:S^\lambda]&=\frac1{n!}\sum_{g\in S_n}h_d(g)\chi^\lambda(g),\\
\dim\Hom_{S_n}(\Sym^d\Vnr,\Vns)&=\frac1{n!}\sum_{g\in S_n}h_d(g)f_s(g).
\end{align*}
Both sums depend on $g$ only through its cycle type.
\end{theorem}

\begin{proof}
A subset is fixed by $g$ if and only if it is a union of cycles of $g$, which gives the generating-function expression for $f_r$.

By \cref{prop:sym-mset}, $\chi_{\Sym^d\Vnr}(g)$ is the number of $g$-invariant multisets of size $d$ on $\Xnr$, that is, of multiplicity functions constant on the cycles of $g$ acting on $\Xnr$. Hence
$$
\sum_{d\ge0}\chi_{\Sym^d\Vnr}(g)t^d=\prod_{c\in\Cyc(g\curvearrowright\Xnr)}\bigl(1-t^{\abs c}\bigr)^{-1},
$$
and taking logarithms gives $\sum_{m\ge1}\frac{t^m}m\abs{\Xnr^{g^m}}$, since a point of $\Xnr$ is fixed by $g^m$ exactly when its cycle length divides $m$. As $\abs{\Xnr^{g^m}}=f_r(g^m)$, this shows $\chi_{\Sym^d\Vnr}(g)=h_d(g)$, which is \cref{prop:molien} for the permutation representation $\Vnr$. The second formula is the first with $\chi_{\Vns}=f_s$.
\end{proof}

The recursion for $h_d$ gives a direct computational procedure.

\subsection{Support and stability}\label{sec:subset-stability}
This supplementary subsection records how the irreducible support behaves as $n$ grows; it is not needed for the coordinatewise layer calculations above.

\begin{theorem}\label{thm:stability}
    \begin{enumerate}
    \item If $S^\lambda$ occurs in $P_\cH(n)$ then $\lambda_1\ge n-v(\cH)$. Hence every constituent of $\Sym^d\Vnr$ satisfies $n-\lambda_1\le rd$.
    \item Fix $r,d$ and a partition $\nu\vdash s$, and put $\lambda[n]=(n-s,\nu)$. For $n$ such that $\lambda[n]$ is a partition, $m_\nu(n):=[\Sym^d\Vnr:S^{\lambda[n]}]$ vanishes for all $n$ if $s>rd$, and is independent of $n$ for $n\ge s+rd$ if $s\le rd$.
    \end{enumerate}
\end{theorem}
\begin{proof}
(a) Since $\Aut(\cH)\times S_{n-v}\supseteq S_{n-v}$ we get $[P_\cH(n):S^\lambda]\le\dim(S^\lambda)^{S_{n-v}}$, and iterated branching gives $(S^\lambda)^{S_m}\ne0$ if and only if $\lambda_1\ge m$.

(b) The first claim is (a). For the second, take $n\ge rd$ so that the type set is fixed. Write $\Ind_{\Aut(\cH)}^{S_v}\one=\bigoplus_\mu c_\mu S^\mu$ and induce in stages; the contribution of $\cH$ to the multiplicity of $S^{\lambda[n]}$ is $\sum_\mu c_\mu\,c^{\lambda[n]}_{\mu,(n-v)}$, and by Pieri's rule the Littlewood--Richardson coefficient $c^{\lambda[n]}_{\mu,(n-v)}$ is $1$ exactly when $\lambda[n]/\mu$ is a horizontal strip, that is when $\lambda[n]_i\ge\mu_i\ge\lambda[n]_{i+1}$ for all $i$, and $0$ otherwise.

Only one of those inequalities involves $n$, namely $n-s\ge\mu_1$. Since $\mu_1\le v\le rd$, it holds for every $\mu$ once $n\ge s+rd$, so the multiplicity is constant from there on.
\end{proof}

In the language of representation stability, the sequence $n\mapsto\Sym^d\Vnr$ is a finitely generated FI-module, which also implies eventual constancy \cite{CEF2015}. The argument above gives an explicit sufficient range for each padded Specht multiplicity. The sharpness statement in \cref{prop:sharp} concerns the full space of maps into $s$-subset features; it does not assert that every individual Specht multiplicity first stabilises at this bound.

\bibliographystyle{alpha}
{\raggedright
\def\bibliofont{\small}
\bibliography{references}
}
\end{document}